\documentclass[11pt]{article}

\usepackage{amsmath}          
\usepackage{amssymb}  
\usepackage{amsmath,amsthm} 
\usepackage{enumerate}
\usepackage{stmaryrd}
\usepackage{url} 
\usepackage{a4wide}

\usepackage[all]{xy}

\usepackage{esvect}      

\usepackage{scalerel}

   \usepackage{bbding}        

\usepackage{mathrsfs}       

\usepackage[round]{natbib}       

\usepackage{times}

\usepackage{mathbbol}    

\newtheorem{theorem}{Theorem}[section]   
\newtheorem{corollary}[theorem]{Corollary}  
\newtheorem{lemma}[theorem]{Lemma}

\DeclareSymbolFont{symbolsC}{U}{txsyc}{m}{n}
\DeclareMathSymbol{\strictif}{\mathrel}{symbolsC}{74}              

\newcommand{\mb}[1]{\mathbf{#1}}
   \newcommand{\A}{\mb{A}}
   \newcommand{\B}{\mb{B}}
      \newcommand{\bH}{\mb{H}}
         \newcommand{\bS}{\mb{S}}
            \newcommand{\bP}{\mb{P}}
   \newcommand{\Pu}{\mb{P}_U}
        
\newcommand{\CC}{\mathcal{C}}  
\newcommand{\K}{\mathcal{K}}  
\newcommand{\V}{\mathcal{V}}  

\def\a{\alpha}
\def\b{\beta}

\def\om{\omega}
\def\sig{\sigma}

\def\ab #1{|#1|}

\def\bo{\mathop{\scalerel*{\Box}{gX}}}           
\def\di{\mathord{\scalerel*{\Diamond}{gX}}}     

\def\C{\mathbb{C}}

\def\CC{\mathscr{C}}

\def\cl{\mathop{\mathsf{Cl}}}   

\def\de{\mathop{\mathsf{De}}}   

\def\D{\mathbb{D}}

\def\F{\mathcal{F}}

\renewcommand{\ge}{\geqslant}       
\renewcommand{\le}{\leqslant} 

\def\int{\mathop{\sf Int}} 

\def\M{\mathcal{M}}

\def\liff{\enspace \text{iff}\enspace}

\def\P{\mathbb{P}}

\def\ph{\varphi}

\def\<{\langle}
\def\>{\rangle}

\def\sees{\rightsquigarrow}

\def\sub{\subseteq}

\def\Var{\mathop{\mathsf{Var}}} 

\def\Vec#1#2{#1_0,\ldots{},#1_{#2}}

\def\ws{\mbox{E}}

\begin{document}

\title{Modal Logics that Bound the Circumference\\ of Transitive Frames.}
\author{Robert Goldblatt\thanks{School of Mathematics and Statistics, Victoria University of Wellington, New Zealand.
{\tt sms.vuw.ac.nz/\~{}rob}}}

\date{}    
\maketitle


\abstract{
For each natural number  $n$ we study the modal logic determined by the class of transitive Kripke frames in which there are no cycles of length greater than $n$ and no strictly ascending chains. The case $n=0$ is the G\"odel-L\"ob provability logic. Each logic is axiomatised by adding a single axiom to K4, and is shown to have the finite model property and be decidable.
\\
\strut\quad We then consider a number of extensions of these logics, including restricting to reflexive frames to obtain a corresponding sequence of extensions of S4. When $n=1$, this gives  the famous logic of Grzegorczyk, known as S4Grz, which is the strongest modal companion to intuitionistic propositional logic.
A topological semantic analysis shows that the $n$-th member of the sequence of extensions of S4 is the logic of 
hereditarily $n+1$-irresolvable spaces when the modality $\Diamond$ is interpreted as the topological closure operation.
We also study the definability of this class of spaces under the interpretation of $\Diamond$ as the derived set (of limit points) operation.
\\
\strut\quad
The variety of modal algebras validating the $n$-th logic is shown to be generated by the powerset algebras of the finite frames with cycle length bounded by $n$. Moreover each algebra in the variety is a model of the universal theory of the finite ones, and so is embeddable into an ultraproduct of them. }



\section{Algebraic Logic and Logical Algebra}      
The field of algebraic logic has been described as having two main aspects (see the introductions to \citealt{daig:stud74} and \citealt*{andr:alge01}). One is the study of algebras arising from logical ideas. The other is the study of logical questions by algebraic methods.

Both aspects are well exemplified in the profound research of Hajnal Andr\'eka and Istv\'an N\'emeti.
Together, and in collaboration with many colleagues, they have created a prodigious body of literature about Boolean algebras, cylindric algebras, polyadic algebras, relation algebras, fork algebras, modal algebras, dynamic algebras, Kleene algebras and others; with applications to questions of definability, axiomatizability, interpolation, omitting types, decidability etc.\ for a range of logics.

Concerning the first aspect, there is no restriction on the methods that may be used to study abstract algebras. Often the work is algebraic, but it may also involve, say, topology or set theory. Or logic itself. The study of algebraic questions by logical methods, a kind of converse to the first second aspect, might be called \emph{logical algebra}.

One of the aims of the present paper is to provide an illustration of logical algebra at work. In the final section we show that some varieties of modal algebras, built from certain finite graphs with bounded circumference, have the property that each member of the variety is embeddable into an ultraproduct of finite members. The logical proof of this structural result involves an adaptation of a construction developed to show that certain modal logics have the finite model property under their Kripke semantics, as well as an analysis of the behaviour of the universal sentences satisfied by the algebras involved.

The initial impetus for this study came from reflection on a property of the well known modal logic of Grzegorczyk, which is characterised by the class of finite partially ordered Kripke frames. A partial order can be described as a quasi-order (reflexive transitive relation) that has circumference equal to 1, where the circumference is the longest length of any cycle. This suggests a natural question: what modal logics are characterised by frames with circumference at most $n$ for an arbitrary natural number $n$? Dropping reflexivity and considering transitive frames, the answer is already known for two cases. For $n=0$ it is the G\"odel-L\"ob modal logic of provability, and for $n=1$ it is a version of Grzegorczyk's logic without reflexivity. Here we will provide a systematic answer for all $n$, giving in each case an axiomatisation of the logic concerned and showing it has the finite model property. We then discuss topological semantics for these logics, and finally turn to  algebra and take up the  matters mentioned in the previous paragraph.

The next section provides more background  on these ideas, as preparation for the technical work to follow.

\section{Grzegorczyk and L\"ob}

\citet{grze:some67} defined a modal logic, which he called G, by adding to S4 the axiom
\begin{equation} \label{Gaxiom}
((p\strictif\bo q)\strictif\bo q)\land ((\neg p\strictif\bo q)\strictif\bo q))\strictif\bo q,
\end{equation}
where $\strictif$ denotes  strict implication, i.e.\ $\ph\strictif\psi$ is $\bo(\ph\to \psi)$. He showed that G is a modal companion to intuitionistic propositional logic, meaning that the latter is embedded conservatively into G by the G\"odel-McKinsey-Tarski translation.

A few years earlier, \citet{sobo:fami64} had defined a logic K1.1 by adding to S4 the axiom
$$
((p \strictif \bo p) \strictif p)) \strictif p,
$$
which he called J1. This was an adaptation of 
$$
((p \strictif \bo p) \strictif  p)) \to(\di\bo p \to p),
$$
which was in turn a simplification by Geach of  
$$
((p \strictif \bo p) \strictif \bo p)) \to(\di\bo p \to\bo p),
$$ 
which had been discovered in 1958 by Dummett as an example of a formula that is not a theorem of S4.3 but is valid under Prior's Diodorean temporal interpretation of necessity in discrete linear time  (see \citealp[p.139]{prio:tens62} and \citealp[p.29]{prio:past67}).

\citet{sobo:cert70} showed that K1.1 is weaker than G, by deriving J1 from Grzegorczyk's G-axiom \eqref{Gaxiom}. He raised the issue of whether K1.1 was \emph{strictly} weaker, suggesting that this was `very probable'. However 
\citet[Section II.3]{sege:essa71} proved that K1.1 and G are the same logic, by showing that K1.1 is determined by the class of finite partially ordered Kripke frames and observing that  \eqref{Gaxiom} is valid in all such frames, hence derivable in K1.1.
Segerberg axiomatised K1.1 as S4 plus
\begin{equation}  \label{grz}
\bo(\bo(p\to\bo p)\to p)\to p,
\end{equation}
which is equivalent to J1 over S4. He gave the name `Grz' to axiom \eqref{grz} in honour of Grzegorczyk, 
and  K1.1/G has been known ever since as S4Grz.

The difference between S4 and S4Grz can be understood in terms of the distinction between \emph{quasi}-ordered and \emph{partially} ordered frames. A quasi-order $(W,R)$ is a reflexive transitive relation $R$ on a set $W$. It is a partial order if in addition it is \emph{antisymmetric}: $xRyRx$ implies $x=y$. The condition `$xRyRx$' defines an equivalence relation on $W$ whose equivalence classes are known as \emph{clusters}.  $R$ is universal on a given cluster and maximally so.
It lifts to a partial order of the set of clusters by specifying that $CRC'$ iff $xRy$ where $x$ is any member of cluster $C$ and $y$ is any member of cluster $C'$. The original relation $R$ is a partial order iff each cluster contains just a single element.

In these terms, S4 has a number of characterisations. It is the logic of all quasi-orders, of all partial orders, and of all finite quasi-orders.  But it is not the logic of all  \emph{finite partial} orders since, as mentioned above, that logic is S4Grz. The Grz-axiom \eqref{grz} is valid in any finite partial order, but invalid in some infinite partial orders. The precise situation is that Grz is valid in a frame $(W,R)$ iff it is a partial order that has no strictly ascending chains, i.e.\ no sequences $x_0R\cdots Rx_nRx_{n+1}R\cdots$ such that  $x_{n+1}{ R}x_n$ fails for all $n$. A finite quasi-order has no such chains, so validates Grz iff it is antisymmetric.

For $n\ge 1$, an \emph{$n$-cycle} is given by a sequence  $x_0,\dots,x_{n-1}$ of  $n$ distinct points that have
 $x_0R\cdots Rx_{n-1}Rx_{0}$.  The points of a cycle all belong to the same cluster, and in a finite frame the length of a longest cycle is equal to the size of a largest cluster. This maximum length/size is the \emph{circumference} of the frame.
 
Our interest in this paper is in relaxing the antisymmetry property of partial orders that constrains clusters to be singletons and cycles to be of length 1. What logic results if we allow cycles of length  up to two, or three etc.? Also we wish to broaden the context to consider transitive frames that may have irreflexive elements, and clusters that may consist of a single such element. So instead of S4, we work over the weaker K4, which is the logic of all transitive frames. That allows us to admit a circumference of 0, since in a finite transitive \emph{irreflexive} frame there are no cycles at all. The logic characterised by such frames has been well studied: it is the smallest normal logic to contain the  L\"ob axiom
\begin{equation}  \label{lobaxiom}
\bo(\bo p\to p)\to\bo p.
\end{equation}
The proof of this fact is also due to  \citet[Section II.2]{sege:essa71}, who called the axiom W and the logic KW. Later \cite{solo:prov76} showed that it is precisely the modal logic that results when $\bo$ is interpreted as expressing provability in first-order Peano arithmetic. It was \cite{lob:solu55} who showed that \eqref{lobaxiom} is valid under this interpretation. The logic is now often called the G\"odel-L\"ob logic, or GL.

For each natural number  $n$ we will define an axiom $\C_n$ which is valid in precisely those transitive frames that have no strictly ascending chains and no cycles (or clusters) with more than $n$ elements. We prove that the logic of this class of frames is axiomatisable as the system K4$\C_n$. The proof uses the familiar technology
of filtration of canonical models and then modification of the filtration to obtain a finite model with the desired properties. Here the modification involves  `breaking up' clusters that contain too many elements, hence destroying cycles that are too long. It establishes that K4$\C_n$ has the finite model property, being the logic of the class of finite transitive frames that have circumference at most $n$, and is a decidable logic. From this we conclude that the logics $\{\mathrm{K4}\C_n:n\ge 0\}$ form a strictly decreasing sequence   of extensions of K4 whose intersection is K4 itself.
The analysis is then adapted to some extensions of K4$\C_n$ obtained by adding the axioms corresponding to seriality, reflexivity and connectedness of the relation $R$.

To indicate the nature of the axioms $\C_n$, we remark first that $\C_0$ is equivalent over all frames to the formula
\begin{equation} \label{dualob}
\di p\to \di(p\land\neg \di p),
\end{equation}
which is itself equivalent  to the  L\"ob axiom \eqref{lobaxiom}.

To explain $\C_1$, observe that Grz is equivalent over S4 to
\begin{equation} \label{grzL}
\text{Grz}_{\Box}: \quad   \bo(\bo(p\to\bo p)\to p)\to \bo p,
\end{equation}
which can be equivalently expressed in terms of $\di$ as
\begin{equation} \label{digrz}
 \di p\to \di(p\land\neg \di(\neg p\land \di p)).
\end{equation}
Replacing $\neg p$ here by a variable which is hypothesised to be incompatible with $p$, we are led  to define $\C_1$ to be the two-variable formula
$$
\bo\nolimits^*\neg(p\land q)\to (\di p\to \di(p\land\neg \di(q\land \di p)))
$$
(where in general $\bo^*\ph$ is $\ph\land\bo\ph$).
This will be shown to be  equivalent to \eqref{digrz}, hence to Grz$_\Box$, over K4. Frames validating the logic K4$\C_1$ have only singleton clusters, some of which may contain irreflexive elements. The finite frames of this kind determine the logic $\mathrm{K4Grz}_{\Box}$, as was  shown by \cite{amer:cutf96} using tableaux techniques, and then by \cite{gabe:topo04} using a filtration method. This logic is equal to K4$\C_1$, while
Grzegorczyk's logic disallows irreflexivity and is equal to S4$\C_1$.
The logic K4$\C_1$ was studied under the name K4.Grz by \cite{gabe:topo04} and \cite{esak:moda06}, and has been called the \emph{weak Grzegorczyk} logic (wGrz) by \cite{lita:non07,lita:cons14}. \cite{amer:cutf96} calls it G$_0$, while
\cite{gore:cut14} call it \emph{Go}.

 Lifting the pattern of $\C_1$ to three variables $p_0,p_1,p_2$, we define $\C_2$ to be
$$
\bo\nolimits^*\big(\bigwedge\nolimits_{ i<j\le 2}\neg(p_i\land p_j) \big) \to
(\di p_0\to \di(p_0\land\neg\di(p_1\land\di(p_2\land\di p_0))) ).    
$$
$\C_n$ extends the pattern to $n+1$ variables, and will be formally defined in Section \ref{secmodels}.

After completing our model-theoretic analysis we turn to topological semantics and show that for all $n\geq 1$, the logic S4$\C_n$ is characterised by validity in the class of all (finite) \emph{hereditarily $n+1$-irresolvable} topological spaces when 
$\di$ is interpreted as the operation of topological closure. Hitherto this characterisation was known only for $n=1$, i.e.\ for S4Grz. We then discuss the interpretation of $\di$ as the derived set operation, assigning to each set its set of limit points, and show that  a space also validates $\C_n$ under this interpretation iff it is hereditarily $n+1$-irresolvable.

In the final section we study the variety $\V_n$ of all modal algebras validating the logic K4$\C_n$. This is generated by its finite members, and indeed by the class $\CC_n^+$ of all powerset algebras of the class $\CC_n$ of all finite transitive frames of circumference at most $n$. Thus $\V_n$ is the class of all models of the set of equations satisfied by $\CC_n^+$. But we show something stronger: $\V_n$ is the class of all models of the set of universal sentences satisfied by $\CC_n^+$. 
It follows that every member of $\V_n$ can be embedded into an ultraproduct of members of $\CC_n^+$. 

\section{Clusters and Cycles}

A \emph{frame} $\F=(W,R)$ is a directed graph, consisting of a binary relation $R$ on a set $W$. A point $x\in W$ is \emph{reflexive} if $xRx$, and \emph{irreflexive} otherwise. If  every member of $W$ is (ir)reflexive, we say that $R$ and $\F$ are  (ir)reflexive. $\F$ is \emph{transitive} when $R$ is a transitive relation.

For the most part we work with transitive frames and informally give $R$ a temporal interpretation, so that if
$xRy$  we may say that $y$ is an  \emph{$R$-successor} of $x$, that $y$ comes \emph{$R$-after} $x$, or is
\emph{$R$-later} than $x$, etc. If $xRy$ but not $yRx$, then $y$ is \emph{strictly} $R$-later than $x$.

In a transitive frame, a \emph{cluster} is a subset $C$ of $W$ that is an equivalence class under the equivalence relation
$
\{(x,y):x=y\text{ or } xRyRx\}.
$
A singleton $\{x\}$ with $x$ irreflexive is a \emph{degenerate} cluster. All other clusters are \emph{non-degenerate}: if $C$ is non-degenerate then it contains no irreflexive points and the relation $R$ is universal on $C$ and maximally so. 
A \emph{simple} cluster is non-degenerate with one element, i.e.\ a singleton $\{x\}$ with $xRx$.

Let $C_x$ be the $R$-cluster containing $x$. Thus $C_x=\{x\}\cup\{y:xRyRx\}$. The relation $R$ lifts to a well-defined relation on the set of clusters   by putting $C_xRC_y$ iff $xRy$. This relation is transitive and \emph{antisymmetric}, for if 
$C_xRC_yRC_x$, then $xRyRx$ and so $C_x=C_y$. A cluster $C_x$ is \emph{final} if it is maximal in this ordering, i.e.\ there is no cluster $C\ne C_x$ with $C_xR C$. This is equivalent to requiring that $xRy$ implies $yRx$.

We distinguish between finite and infinite sequences of $R$-related points.
An \emph{$R$-path} in a frame is a finite sequence $x_0,\dots,x_n$  of (not necessarily distinct) points from $W$ with $x_mRx_{m+1}$ for all $m<n$. 
An \emph{ascending $R$-chain} in a frame is an infinite sequence $\{x_m:m<\om\}$ of (not necessarily distinct) points from $W$ with $x_mRx_{m+1}$ for all $m<\om$. If $R$ is transitive, this implies $x_mRx_k$ whenever $m<k$. The chain is \emph{strictly ascending} if not $x_{m+1}{R}x_{m}$ for all $m$, hence for transitive $R$, not $x_{k}{R}x_{m}$ whenever $m<k$. Observe that if $x$ is a reflexive point then the constant infinite sequence $x,\dots,x,\dots$ is an ascending $R$-chain that is not strict.
In a transitive frame, a strictly ascending  chain has all its terms $x_m$ being  pairwise distinct, so there are infinitely many of them.

\begin{lemma} \label{const}
The following are equivalent for any transitive frame $\F=(W,R)$:
\begin{enumerate}[\rm(1)]
\item 
There are no strictly ascending chains of points in $\F$.
\item
Any ascending chain $C_0RC_1R\cdots\cdots$   of $R$-clusters is \textbf{ultimately constant} in the sense that there exists an $m$ such that $C_m=C_k$ for all $k>m$.
\end{enumerate}
\end{lemma}
\begin{proof}
Suppose (1) fails, and there is a strictly ascending $R$-chain $\{x_m:m<\om\}$. Then $C_{x_m}R\,C_{x_{m+1}}$ and not 
$C_{x_{m+1}}R\, C_{x_m}$ for all $m$, so the cluster chain
$\{C_{x_m}:m<\om\}$ is strictly ascending and hence not ultimately constant, showing that (2) fails.

Conversely if (2) fails, there is an ascending cluster chain $\{C_m:m<\om\}$ that is not ultimately constant. So for all $m$ there exists a $k>m$ such that $C_m\ne C_k$, and hence not $C_kR C_m$ as $C_mR C_k$ and R is antisymmetric on clusters. Using this we can pick out a subsequence $\{C_{fm}:m<\om\}$ that is strictly ascending. Then choosing $x_m\in C_{fm}$ for all $m$ gives a  chain of points  $\{x_m:m<\om\}$ that is strictly ascending, showing that (1) fails.
\end{proof}

A \emph{cycle of length $n\ge 1$}, or \emph{$n$-cycle}, is a sequence  $x_0,\dots,x_{n-1}$ of $n$  \emph{distinct} points such that $x_0,\dots,x_{n-1},x_0$ is an $R$-path. There are no 0-cycles.
A 1-cycle is given by a single point $x_0$ having $x_0Rx_0$.

Adopting terminology from graph theory, we define the \emph{circumference} of frame $\F$ to be the supremum of the set of all lengths of cycles in $\F$.  In particular $\F$ has circumference $0$ iff it has no cycles, a property implying that $\F$ has no reflexive points. In a finite frame with non-zero circumference, since there are finitely many cycles the  circumference is the length of a \emph{longest} one.

In a transitive frame, the points of any cycle  are $R$-related to each other and are reflexive, and all belong to the same non-degenerate cluster. It follows that the circumference is $0$ iff the frame is irreflexive.
Moreover,  any finite non-empty subset of a non-degenerate cluster can be arranged (arbitrarily) into a cycle. Thus for $n\ge 1$, a frame has a cycle of length $n$ iff it has a non-degenerate cluster of size at least $n$. So  a non-zero circumference of a finite transitive frame is equal to the size of a largest non-degenerate cluster.

\section{Models and Valid Schemes}  \label{secmodels}

In the standard language of propositional modal logic, formulas $\ph,\psi,\dots$ are constructed from some denumerable set $\Var$ of propositional variables by the  Boolean connectives $\top$, $\neg$, $\land$ and the unary modality $\bo$. The other Boolean connectives $\bot$, $\lor$, $\to$, $\leftrightarrow$ are introduced as the usual abbreviations, and the dual modality $\di$ is defined to be $\neg\bo\neg$.
We write $\bo^*\ph$ as an abbreviation of the formula $\ph\land\bo\ph$, and $\di^*\ph$ for $\ph\lor\di\ph$.

We now describe the \emph{Kripke semantics}, or \emph{relational semantics}, for this language. A \emph{ model} $\M=(W,R,V)$  on a frame $(W,R)$ is given by a valuation function $V$ assigning to each variable $p\in\Var$ a subset $V(p)$ of $W$, thought of as the set of points of $W$ at which $p$ is true.  The \emph{truth relation }
$\M,x\models\varphi$ of a formula $\varphi$  being \emph{true at $x$ in $\M$} is defined by an induction on the formation of $\ph$ as follows:
\begin{itemize}
\item $\M,x\models p$ iff $x\in V(p)$, for $p\in\Var$.

\item $\M,x\models\top$.

\item $\M,x\models\neg\varphi$ iff $\M,x\not\models\varphi$ (i.e.\ not $\M,x\models\varphi$).

\item $\M,x\models\varphi\wedge\psi$ iff $\M,x\models\varphi$ and $\M,x\models\psi$.

\item $\M,x\models\bo\varphi$ iff $\M,y\models\varphi$ for every $y\in W$ such that $xRy$.
\end{itemize}
Consequently,   where $R^*=R\cup\{(x,x):x\in W\}$, the reflexive closure of $R$, we have
\begin{itemize}
\item   \label{item:semdi}
$\M,x\models\di\ph$ iff $\M,y\models\varphi$ for some  $y\in W$ such that $xRy$.

\item  $\M,x\models\bo^*\ph$ iff $\M,y\models\varphi$ for every  $y\in W$ such that $xR^*y$.

\item  $\M,x\models\di^*\ph$ iff $\M,y\models\varphi$ for some  $y\in W$ such that $xR^*y$.
 \end{itemize} 

A model $\M$ assigns to each formula $\ph$ the \emph{truth set}  $\M\ph=\{x\in W:\M,x\models\ph\}$. (The semantics could have given by  defining truth sets inductively, starting with $\M p=V(p)$.)
We say that $\ph$ is \emph{true in model $\M$}, written  $\M\models\ph$, if it is true at all points in $\M$, i,e.\
$\M\ph=W$. We call  $\ph$ \emph{valid in frame $\F$}, written $\F\models\ph$,  if it is true in all models on $\F$. 
We may also write $\M\models\Phi$, or $\F\models\Phi$, to indicate that every member of a set $\Phi$ of formulas is true in $\M$, or valid in $\F$.

\medskip
Given formulas $\Vec{\ph}{n}$,
define the formula $\P_n(\Vec{\ph}{n})$ to be
$$
\di(\ph_1\land\di(\ph_2\land\cdots \land\di(\ph_n\land\di \ph_0))\cdots)
$$
provided that $n\ge 1$. For the case $n=0$, put $\P_0(\ph_0)=\di\ph_0$.
This definition can made more formal  by  inductively defining  a sequence $\{\P_n:n< \om\}$ of operations on formulas, with $\P_n$ being $n+1$-ary.  $\P_0$ is as just given, and  for $n>0$ we inductively  put 
$$
\P_n(\ph_0,\ph_1,\dots,\ph_n)=\di(\ph_1\land \P_{n-1}(\ph_0,\ph_2,\dots,\ph_n)).
$$
Then the next result follows readily from the properties of the truth relation.

\begin{lemma} \label{Rpath}
In any model $\M$ on any frame, $\M,x_0\models\P_n(\Vec{\ph}{n})$ iff there is an $R$-path $x_0R\cdots Rx_{n+1}$ such that $\M,x_i\models\ph_i$ for $1\le i\le n$ and $\M,x_{n+1}\models\ph_0$.

\end{lemma}

Let $\D_n(\Vec{\ph}{n})$ be $\bigwedge_{ i<j\le n}\neg(\ph_i\land \ph_j)$. For $n=0$ this is the empty conjunction, which we take to be the constant tautology $\top$.
Define $\C_n$ to be the scheme
$$
\bo\nolimits^*\D_n(\Vec{\ph}{n}) \to(\di \ph_0 \to
\di(\ph_0\land\neg\P_n(\Vec{\ph}{n}) ).
$$
In other words, $\C_n$ is the set of all uniform substitution instances of 
$$
\bo\nolimits^*\D_n(\Vec{p}{n}) \to(\di p_0 \to
\di(p_0\land\neg\P_n(\Vec{p}{n}) ),
$$
where $\Vec{p}{n}$ are variables.

\begin{theorem} \label{sound}
Let $\F$ be any transitive frame, and $n\ge 0$.

\begin{enumerate}[\rm 1.]
\item   \label{1(1)}
$\F\models\C_n$ iff $\F$  has circumference at most $n$ and has  no strictly ascending chains.
\item
If $\F$ is \textbf{finite}, then $\F\models\C_n$  iff $\F$  has circumference at most $n$.
\end{enumerate}
\end{theorem} 
\begin{proof}
Fix a list $\Vec{p}{n}$ of variables and abbreviate $\D_n(\Vec{p}{n})$ to $\D_n$ and $\P_n(\Vec{p}{n})$ to $\P_n$. Then the scheme $\C_n$ is valid in $\F$ iff its instance
\begin{equation}  \label{DP}
\bo\nolimits^*\D_n\to(\di p_0 \to
\di(p_0\land\neg\P_n ))
\end{equation}
is valid, since validity in a frame preserves uniform substitution of formulas for variables.

(1).  Assume $\F\models\C_n$. Then we show that $\F$ has circumference at most $n$ and   no strictly ascending chains. For, if that were to fail there would be two possible cases, the first being that $\F$ has  circumference greater than  $n$, so has a cycle  with at least $n+1$  elements, say $x_0,\dots,x_n$. If $n=0$ then $\D_n=\top$. If $n\ge 1$, take a model on $\F$  having  $V(p_i)=\{x_i\}$ for all $i\le n$. 
 The $x_i$'s are distinct and each formula $\neg(p_i\land p_j)$ with $i<j\le n$ is true at every point in the model, hence so is $\D_n$. So whatever the value of $n$,  $\D_n$ is true everywhere, and therefore so is $\bo^*\D_n$. By transitivity all points of the cycle are  $R$-related to each other, hence to $x_0$, including $x_0$ itself. Therefore $\di p_0$ is true at $x_0$.
By Lemma \ref{Rpath}, the $R$-path $x_0Rx_1R\cdots x_nRx_{0}$ ensures that $\P_n$ is true at $x_0$. Hence as $p_0$ is true only at $x_0$, $p_0\land\neg\P_n$ is false everywhere, therefore so is $\di(p_0\land\neg\P_n)$. Altogether these facts imply that the instance \eqref{DP} of   $\C_n$ is false at $x_0$ (in fact at every point of the cycle), contradicting the assumption that $\F\models\C_n$.

The second case is that $\F$ has a strictly ascending $R$-chain $\{x_m:m<\om\}$.   Then take a model on $\F$  having  
$V(p_i)=\{x_m: m\equiv i\bmod{(n+1)}\}$ for all $i\le n$. Since the points $x_m$ of the chain are all distinct, the sets 
$V(p_i)$ are all pairwise disjoint,  so each formula $\neg(p_i\land p_j)$ is true everywhere, hence so is $\bo^*\D_n$.
Since each congruence class mod $n+1$ is cofinal in $\om$, each set $V(p_i)$ is cofinal in the chain, i.e.\ for all $k<\om$ there is an $m>k$ with $x_kRx_m\in V(p_i)$. In particular this implies that $\di p_0$ is true at every point of the chain. 
Now if $p_0$ is true at point $x_m$, then the $R$-path
$x_mRx_{m+1}R\cdots  x_{m+n}Rx_{m+(n+1)}$  has $p_i$ true at $x_{m+i}$ for  $1\le i\le n$ and $p_0$ true at 
$x_{m+(n+1)}$, so $\P_n$ is true at $x_m$ by Lemma \ref{Rpath}. Since $p_0$ is true only at points of the chain,
it follows that $\di(p_0\land\neg\P_n )$ is false everywhere. Altogether then, \eqref{DP} is false at all points of the chain, again contradicting $\F\models\C_n$. The contradictions in both cases ensure that if $\F\models\C_n$ then $\F$ has circumference at most $n$ and  no strictly ascending chains.

Finally, assume that  $\F$  has circumference at most $n$ and  no strictly ascending chains. Suppose then, for the sake of contradiction, that $\C_n$ is not valid in $\F$. Hence  \eqref{DP} is not valid and so is false at some $x$ in some model on $\F$. Working in that model, $\di p_0$ and $\bo^*\D_n$ are true at $x$, so $p_0$ is true at some $x_0$ with $xRx_0$, and the formulas $\neg(p_i\land p_j)$ are all true throughout $\{y\in W: xR^*y\}$; while 
$\di(p_0\land\neg\P_n )$ is false at $x$. As $xRx_0$,  $p_0\land\neg\P_n$ is false at $x_0$. Since $p_0$ is true at  $x_0$,  this implies that $\P_n$ is true at $x_0$. Hence by Lemma \ref{Rpath}, there is an $R$-path $x_0Rx_1R\cdots Rx_{n+1}$ such that $p_i$ is true at $x_i$ for $1\le i\le n$ and $p_0$ is true at $x_{n+1}$. The argument then repeats: 
 by transitivity $xRx_{n+1}$, so since $\di(p_0\land\neg\P_n )$ is false at $x$,  $p_0\land\neg\P_n$ is false at $x_{n+1}$, and hence $\P_n$ is true at $x_{n+1}$. So by Lemma \ref{Rpath} again, there is an $R$-path $x_{n+1}Rx_{n+2}R\cdots Rx_{2(n+1)}$ such that $p_i$ is true at $x_{n+1+i}$ for $1\le i\le n$ and $p_0$ is true at $x_{2(n+1)}$.

Iterating this construction ad infinitum, we generate an ascending $R$-chain
$\{x_m: m<\om\}$ of points of $W$ such that for each $i\le n$, $p_i$ is true at $x_m$ iff $m\equiv i\bmod{(n+1)}$. Hence $p_i$ is true cofinally along the chain.
By assumption there are no \emph{strictly} ascending chains, so by Lemma \ref{const} the ascending cluster chain
$\{\ab{x_m}: m<\om\}$ is ultimately constant. It follows that the point chain cannot continue moving forward into a `later' cluster forever, so some cluster $C$ in the cluster chain must contain some tail 
$\{x_m: m\ge k\}$ of the point chain.  Then $x_k,x_{k+1}\in C$ and $x_kRx_{k+1}$, so $C$ is a non-degenerate cluster.
The tail $\{x_m: m\ge k\}$ contains points at which each of $p_0,\dots,p_n$ are true, by the cofinality of the truth of these $p_i$'s. But by the truth of $\bo^*\D_n$ at $x$, no two of these variables are ever true at the same point of the point chain. Hence the cluster $C$ contains at least $n+1$ distinct points, which form an $n+1$-cycle. This contradicts the assumption that $\F$ has circumference at most $n$. The contradiction forces us to conclude that $\F\models\C_n$.

(2).
This follows immediately from (1), as a finite transitive frame cannot have any strictly ascending chains.
\end{proof}
The cases $n=0,1$ of this theorem for $\C_n$ are essentially known. $\C_0$ is
$$
\bo\nolimits^*\top\to(\di \ph_0 \to
\di(\ph_0\land\neg\di\ph_0 )).
$$
That is valid in the same frames as $(\di \ph_0 \to\di(\ph_0\land\neg\di\ph_0 ))$,  an equivalent form of the L\"ob axiom
$$
\bo(\bo\ph_0\to\ph_0)\to\bo\ph_0.
$$
But it is well known that the L\"ob axiom is valid in a frame $\F$ iff $\F$ is transitive and has no ascending chains (see 
\citealp[p.~75]{bool:logi93}   or    \citealp[Example 3.9]{blac:moda01}).
Now a transitive frame has circumference 0 iff it is irreflexive, and in a transitive irreflexive frame every ascending chain is strictly ascending. From these facts
it can be seen that a transitive frame has no ascending chains iff it has circumference 0 and no strictly ascending chains.

For $n=1$, $\C_1$ is valid in the same transitive frames as the Grz-variant Grz$_\Box$ of \eqref{grzL} (see Theorem \ref{C12} below). But a transitive frame validates Grz$_\Box$ iff it has no ascending chains  $x_0Rx_1R\cdots $ with $x_n\ne x_{n+1}$ for all $n$
\cite[Lemma 1.1]{amer:cutf96}. The latter condition prevents there being any clusters with more than one element, ensuring that the circumference is at most 1. Thus it can be seen that a transitive frame validates Grz$_\Box$ iff  it has circumference at most 1 and no strictly ascending chains.

\section{Logics and Canonical Models}  \label{seclogics}

A \emph{normal logic} is any set $L$ of formulas that
 includes all tautologies and all instances of the scheme
\begin{description}
\item[K:]
$\bo(\ph\to\psi)\to(\bo\ph\to\bo\psi)$,
\end{description}
and whose rules include  modus ponens and $\Box$-generalisation (from $\ph$ infer $\bo\ph$).  
The set of all formulas valid in some given class of frames is a normal logic. The smallest normal logic, known as K, consists of the formulas that are valid in all frames. We mostly use the standard naming convention for logics that if $\Phi_1,\dots,\Phi_n$ is a sequence of sets of formulas then K$\Phi_1\cdots\Phi_n$ denotes the smallest normal logic that includes $\Phi_1\cup\cdots\cup\Phi_n$.

The members of a logic $L$ may be referred to as the \emph{$L$-theorems}. A formula $\ph$ is \emph{$L$-consistent} if $\neg\ph$ is not an $L$-theorem, and a set of formulas is $L$-consistent iff the conjunction of any of its finite subsets is $L$-consistent. A formula is an $L$-theorem iff it belongs to every \emph{maximally} $L$-consistent set of formulas.

A normal logic $L$ has the \emph{canonical frame}
 $\F_L=(W_L,R_L)$, where $W_L$ is the set of maximally $L$-consistent sets of formulas, and $xR_Ly$  iff $\{\ph:\bo\ph\in x\}\sub y$ iff  $\{\di\ph:\ph\in y\}\sub x$. 

By standard canonical frame theory (e.g.\ \citealt[Chapter 4]{blac:moda01} or \citealt[Chapter 3]{gold:logi92}),  we have that for all formulas $\ph$ and all $x\in W_L$:
\begin{eqnarray}
\label{Boxcan}
\bo\ph\in x &&\text{iff\quad  for all } y\in W_L, \ xR_Ly \text{ implies }\ph\in y.
\end{eqnarray}
The \emph{canonical model} $\M_L$ on $\F_L$ has $V(p)=\{x\in W_L:p\in x\}$ for all $p\in\Var$. With the help of \eqref{Boxcan} it can be shown that it satisfies
\begin{equation}\label{truthlemma}
\M_L,x\models \ph \enspace \text{iff} \enspace \ph\in x,
\end{equation}
a result known as the  \emph{Truth Lemma} for $\M_L$. It implies that the formulas that are true in $\M_L$ are precisely the $L$-theorems, since these are precisely the formulas that belong to every member of $W_L$.
Thus $\M_L\models\ph$ iff $\ph\in L$.

A logic $L$ is \emph{transitive}  if it includes all instances of the scheme
\begin{description}
\item[4:]
$\bo\ph\to\bo\bo\ph$.
\end{description}
The set of formulas valid in some class of transitive frames is a transitive normal logic.
The smallest transitive normal logic is known as K4. Its theorems are precisely the formulas that are valid in all transitive frames.
If a logic $L$ extends K4, equivalently if $\M_L\models 4$, then the relation $R_L$ of its canonical frame is transitive.

Scheme 4 has the equivalent dual form  $\di\di\ph\to\di\ph$. This has the weaker variant
\begin{description}
\item[w4:]
$\di\di\ph\to\di^*\ph$,
\end{description}
i.e.\ $\di\di\ph\to\ph\lor\di\ph$. It plays a significant role in the topological semantics of Section \ref{sectop}. The theorems of Kw4, the smallest normal logic to include w4, are precisely the formulas that are valid in all  frames that are \emph{weakly transitive} in the sense that $xRyRz$ implies $xR^*z$.
If $\M_L\models \text{w4}$, then $R_L$ is weakly transitive.

Since $\bo^*\top$ is a theorem of any normal logic, the scheme $\C_0$ is deductively equivalent over K to the dual form \eqref{dualob} of L\"ob's axiom. But scheme 4 is derivable from L\"ob's axiom over K (see \citealt[p.~11]{bool:logi93}),
so K4$\C_0=$K$\C_0=$GL.

For $n=1$ we have:

 \begin{theorem}  \label{C12}
 The scheme
$\C_1$ is deductively equivalent over Kw4 to the scheme
\begin{equation} \label{digrz0}
\di \ph_0\to \di(\ph_0\land\neg \di(\neg \ph_0\land \di \ph_0))
\end{equation}
that is itself deductively equivalent over K to the Grz-variant Grz$_\Box$ of \eqref{grzL}.
\end{theorem}
\begin{proof}
For any $\ph_0$, the formula $\C_1(\ph_0,\neg \ph_0)$ is
$$
\bo\nolimits^*\neg(\ph_0\land \neg \ph_0)\to (\di \ph_0\to \di(\ph_0\land\neg \di(\neg \ph_0\land \di \ph_0))).
$$
But $\neg(\ph_0\land \neg \ph_0)$ is a tautology, so $\bo^*\neg(\ph_0\land \neg \ph_0)$ is derivable in K, hence can be detached from $\C_1(\ph_0,\neg \ph_0)$ to derive  \eqref{digrz0}.

In the converse direction, for any  $\ph_0$ and  $\ph_1$ the formula
$$
\bo\nolimits^*\neg(\ph_0\land \ph_1)\to \big [\eqref{digrz0} \to (\di \ph_0\to \di(\ph_0\land\neg \di(\ph_1\land \di \ph_0))) \,\big]
$$
can be shown to be valid in all weakly transitive frames, hence is a theorem of  Kw4. Using it and  tautological reasoning, from  \eqref{digrz0} we can derive
$$
\bo\nolimits^*\neg(\ph_0\land \ph_1)\to (\di \ph_0\to \di(\ph_0\land\neg \di(\ph_1\land \di \ph_0))),
$$
which is $\C_1(\ph_0,\ph_1)$.
\end{proof}

We now introduce an apparent weakening of $\C_n$,  defining $\C_n^*$ to be the  scheme
$$
\bo\nolimits^*\D_n(\Vec{\ph}{n}) \to(\di \ph_0 \to
\di^*(\ph_0\land\neg\P_n(\Vec{\ph}{n}) ).
$$
Since  any formula of the form $\di\ph\to\di^*\ph$ is a tautology, $\C_n^*$ is a tautological consequence of $\C_n$, so is included in any logic that includes $\C_n$, and is true in any model in which $\C_n$ is true. In the converse direction we have

\begin{theorem}  \label{w44}
\begin{enumerate}[\rm 1.]
\item 
Let $\M$ be a model with weakly transitive relation $R$. Then $\M\models \C_n^*$ implies $\M\models \C_n$.
\item
$\C_n$ is included in any normal logic that includes w4 and $\C_n^*$.
\end{enumerate}
\end{theorem}

\begin{proof}

(1) Suppose that $\M \not\models \C_n$.
Then there is an instance
\begin{equation}  \label{Cninst2}
\bo\nolimits^*\D_n(\Vec{\ph}{n}) \to(\di \ph_0 \to
\di(\ph_0\land\neg\P_n(\Vec{\ph}{n}) )
\end{equation}
of $\C_n$ that is not true at some point $x$  in $\M$.  Working in $\M$, we have $x\models \bo^*\D_n$ and $x\models \di\ph_0$, but  $x\not\models  \di(\ph_0\land\neg\P_n)$.

If $n=0$, then immediately $x\not\models (\ph_0\land\neg\P_n)$ as $\P_n=\di\ph_0$ and $x\models \di\ph_0$.

If $n\geq 1$ we proceed as follows.
Since $x\models \di\ph_0$, there is some $x_0$ with $xRx_0$ and $x_0\models \ph_0$. Since $x\not\models  \di(\ph_0\land\neg\P_n)$, this gives $x_0\not\models  \ph_0\land\neg\P_n$, hence $x_0\models\P_n$. Thus by Lemma \ref{Rpath}, there is an $R$-path $x_0R\cdots Rx_{n+1}$ such that $x_i\models\ph_i$ for $1\le i\le n$ and $x_{n+1}\models\ph_0$.
Now suppose $x\models \ph_0$. Then $x\not\models \ph_1$ as $x\models \bo^*\D_n$. But  $x_1\models \ph_1$, so we have $xRx_0Rx_1$ and $x\ne x_1$, hence $xRx_1$ by weak transitivity. Thus $x$ can replace $x_0$ to give the $R$-path $xRx_1R\cdots Rx_{n+1}$ demonstrating that $x\models\P_n$.
This proves that $x\models \ph_0$ implies $x\models\P_n$, and therefore that  $x\not\models (\ph_0\land\neg\P_n)$. 

So in any case we have $x\not\models (\ph_0\land\neg\P_n)$.
Since we already have $x\not\models  \di(\ph_0\land\neg\P_n)$, we conclude that
$x\not\models  \di^*(\ph_0\land\neg\P_n)$.
Together with  $x\models \bo^*\D_n$ and $x\models \di\ph_0$, this shows that   $\M \not\models \C_n^*$.

(2). Let $L$ be any normal logic that includes w4 and $\C_n^*$. We apply (1) with $\M$ as the canonical model $\M_L$. Since $w4\sub L$, from the Truth Lemma \ref{truthlemma} we get $\M_L\models \text{w4}$ and so $R_L$ is weakly transitive. Since  $\C_n^*\sub L$ we get $\M_L\models \C_n^*$. Hence by (1),  $\M_L\models \C_n$, so $\C_n\sub L$.
\end{proof}

This theorem implies that 
$$\text{
Kw4$\C_n^*=$Kw4$\C_n$ and  K4$\C_n^*=$K4$\C_n$.}
$$
But for $n=0$ or 1 these are the one logic.  Since   K4$\C_0=$K$\C_0=$GL (above), we have    Kw4$\C_0=$ K4$\C_0$. Also Theorem \ref{C12} gives Kw4$\C_1=$Kw4Grz$_\Box$,  and \citet[Lemma 4.5]{gabe:topo04} has shown that scheme 4 is derivable  from  w4 and Grz$_\Box$ over K, so
 Kw4$\C_1=$K4Grz$_\Box$ $=$ K4$\C_1$.
 
 For $n\geq 2$, Kw4$\C_n$ is strictly weaker than K4$\C_n$. To see this, let $\F$ be an irreflexive frame consisting of two points that are $R$-related to each other but not to themselves. $\F$ is weakly transitive but not transitive, and validates  
 $\C_n$ for $n\geq 2$, hence $\F\models \text{Kw4}\C_n$ but  $\F\not\models \text{K4}\C_n$.

It is straightforward to give a proof-theoretic derivation of $\C_{n+1}$ in K4$\C_n$, showing that $\mathrm{K4}\C_{n+1}\sub \mathrm{K4}\C_{n}$.
We will prove in Theorem \ref{compl} that $\mathrm{K4}\C_{n}$ is characterised by validity in all finite transitive frames of circumference at most $n$. This is already known for $n=0,1$, as mentioned earlier.  K4$\C_0$, the G\"odel-L\"ob logic,  was first shown by \citet{sege:essa71} to be characterised by the class of finite transitive irreflexive frames, which are the finite transitive frames of circumference 0. Also K4$\C_1$, as the logic $\mathrm{K4Grz}_\Box$, was  shown by \cite{amer:cutf96}  to be characterised by the class of finite transitive antisymmetric frames, i.e. those having only singleton clusters, hence circumference at most 1.
We will however include the cases $n=0,1$ in our completeness proof to follow.

\section{Finite Model Property for K4$\C_n$}  \label{fmpn}

Let  $\M=(W,R,V)$ be any model that has transitive $R$ and $\M\models\C_n$, i.e.\ every instance of $\C_n$ is true in $\M$. For example, the canonical model of any normal logic extending K4$\C_n$ has these properties. When working within $\M$ we may sometimes leave out its name and just write $x\models\ph$ when $\M,x\models\ph$. We now set up a filtration of 
$\M$.

Let $\Phi$ be a finite set  of formulas that is closed under subformulas. 
For each $x\in W$ let $x^\Phi$ be the set $\{\ph\in\Phi:x\models\ph\}$ of all members of $\Phi$ that are true at $x$ in $\M$.
An equivalence relation $\sim$ on $W$ is given by putting
$x\sim y$ iff $x^\Phi=y^\Phi$. 
We write  $\ab{x}$ for the equivalence class $\{y\in W:x\sim y\}$, and put  $W_\Phi=\{\ab{x}:x\in W\}.$
The set $W_\Phi$ is finite, because the map $\ab{x}\mapsto x^\Phi$ is a well-defined injection of $W_\Phi$ into the finite powerset of $\Phi$. Thus $W_\Phi$ has size at most $2^{\text{size}\,{\Phi}}$.

Let $\M_\Phi=(W_\Phi,R_\Phi,V_\Phi)$ be the standard transitive  filtration of $\M$ through $\Phi$. Thus  $\ab{x}R_\Phi\ab{y}$ iff $\{\bo \ph,\ph:\bo \ph\in x^\Phi\}\sub {y^\Phi}$, and
$V_\Phi(p)=\{\ab{x}: x\models p\}$ for $p\in\Phi$, while $V_\Phi(p)=\emptyset$  otherwise. The relation $R_\Phi$ is transitive and has the important property that 
\begin{equation} \label{presR}
xRy \quad\text{implies}\quad   \ab{x}R_\Phi\ab{y},
\end{equation}
for all $x,y\in W$.
The\emph{ Filtration Lemma} gives that for all $\ph\in\Phi$ and all $x\in W$,
 \begin{equation}\label{filtlem}
\text{$\M_\Phi,\ab{x}\models \ph$ \enspace iff \enspace  $\M,x\models \ph$.}
 \end{equation}
We will use the fact that for any $R_\Phi$-cluster $C$, and any formula $\di \ph\in\Phi$,
\begin{equation} \label{same}
\text{if $\ab{x},\ab{y}\in C$, then $\M,x\models\di \ph$\enspace iff\enspace$\M,y\models\di \ph$.}
\end{equation}
This follows from the Filtration Lemma, since if $\ab{x}$ and $\ab{y}$ are in the same $R_\Phi$-cluster, then exactly the same formulas of the form $\di \ph$ are true at both of them in $\M_\Phi$.

We will replace $R_\Phi$ by a relation $R'\sub R_\Phi$ in such a way that each $R_\Phi$-cluster $C$ in $\M_\Phi$ is decomposed into  an $R'$-cluster with at most $n$ elements and  (possibly) some singleton (i.e.\ one-element) $R'$-clusters.

We  use letters $\a,\b$ for members of $W_\Phi$. Each such member is a subset of $W$. For each $x\in W$ we write $x\sees\a$ to mean that there is some $y\in\a$ such that $xRy$. This could be read `$x$ can see into $\a$'. We write $x\not\sees\a$ if there is no such $y$.

The next result uses the axiom $\C_n$ to establish a property of $R_\Phi$ that will allow us to refine it into a transitive relation whose clusters have at most $n$ elements.

\begin{lemma} \label{C*}
For any $R_\Phi$-cluster $C$, there is an element $x^*\in W$ such that $\ab{x^*}\in C$ and a subset $C^*\sub C$ such that   $x^* \sees\a$ for all $\a\in C^*$, and for all $y\in W$,  
\begin{equation}  \label{Clast}
\text{if $x^*Ry$ and $\ab{y}\in C$, then $\ab{y}\in C^*$ and $y\sees\a$ for all $\a\in C^*$.}
\end{equation}
Moreover $C^*$ has at most $n$ elements. If $C$ is $R_\Phi$-degenerate then $C^*$ is empty. 
\end{lemma}

\begin{proof}
Take any $R_\Phi$-cluster $C$. 
For each $x \in W$, define  $C(x) = \{\a \in C: x \sees \a\}$.  
Choose $x^*\in W$ such that $\ab{x^*} \in C$ and $C(x^*)$ has least possible size, subject to this.  Then define 
$C^* = C(x^*)$, which immediately ensures that  $x^* \sees\a$ for all $\a\in C^*$. 
To prove \eqref{Clast},  suppose $x^*Ry$ and $\ab{y} \in C$.  Then $\ab{y} \in C^*$ by definition of $C^*$, since ${x^{*}}\sees \ab{y} \in C$. Also, $C(y)$ is a subset of $C(x^*)$ by transitivity of $R$, so by minimality of $C(x^*)$ we get $C(y)=C(x^*)$.  Hence if $\a \in C^*$, then $\a \in C(x^*) = C(y)$, so $y \sees  \a$ as required.

Note that if $\ab{x^*}\in C^*$, then $x^*\sees\ab{x^*}$, therefore $\ab{x^*}R_\Phi\ab{x^*}$ by \eqref{presR}, making $C$ non-degenerate. Thus if $C$ is degenerate, then $\ab{x^*}\notin C^*$, but also $C=\{\ab{x^*}\}$ as $C$ is a singleton containing $\ab{x^*}$, hence $C^*$ is empty as it is a subset of $C$.

It remains to show that $C^*$ has at most $n$ elements. This is where the core role of axiom $\C_n$ is played.
Suppose, for the sake of contradiction, that $C^*$ has $n+1$ distinct members $\a_0,\dots,\a_n$. By standard filtration theory, 
these members are definable as subsets of $\M$, i.e.\ for each $i\le n$ there is a formula $\ph_i$ such that for all $y\in W$,
\begin{equation}\label{distinct}
\M,y\models\ph_i \liff  y\in\a_i \quad(\text{iff } \ab{y}=\a_i).
\end{equation}
If $n\ge 1$, since
the $\a_i$'s are distinct equivalence classes under $\sim$ they are pairwise disjoint, and hence  for all $i< j\le n$, the formula 
$\neg(\ph_i\land\ph_j)$   is true in $\M$ at every $y\in W$, therefore so is  $\D_n(\Vec{\ph}{n})$.
Thus by the semantics of $\bo^*$, 
\begin{equation} \label{ante}
\M,x^*\models\bo\nolimits^* \D_n(\Vec{\ph}{n}).
\end{equation}
But also as $\bo^*\D_0(\ph_o)=\bo^*\top$,   \eqref{ante} holds  as well when $n=0$.

Now since $\a_0\in C^*$ we have $x^*\sees\a_0$, hence using \eqref{distinct} we get $\M,x^*\models\di\ph_0$.
Combining this with \eqref{ante} and the fact that  every instance of $\C_n$ is true in $\M$, we get that
$$
\M,x^*\models\di(\ph_0\land\neg\P_n(\Vec{\ph}{n}) ).
$$
 Hence  there is some $x_0$ with $x^*Rx_0$  and 
\begin{equation} \label{con}
\M,x_0\models\ph_0\land\neg\P_n(\Vec{\ph}{n}).
\end{equation}
Therefore $x_0\models\ph_0$ , so  by \eqref{distinct} $\ab{x_0}=\a_0\in C^*$. 
We will now construct an $R$-path $x_0,\dots,x_n$ with $x_i\in\a_i$, hence by \eqref{distinct}  $x_i\models\ph_i$,  for all $i\le n$. 
If $n=0$ we already have $x_0\in\a_0$ and there is nothing further to do. If $n>0$, then assume inductively that for some $k<n$ we have defined an $R$-path $x_0,\dots ,x_k$ with $x_i\in\a_i$ for all $i\le k$. Since  $x^*Rx_0$, by transitivity we get $x^*Rx_k$.
As  $\ab{x_k}=\a_k\in C$, we get $x_k\sees \a_{k+1}\in C^*$ by \eqref{Clast}, so there is some $x_{k+1}\in\a_{k+1}$ such that $x_kRx_{k+1}$.
That completes the inductive construction of the $R$-path $x_0,\dots,x_n$  with $x_i\in\a_i$ for all $i\le n$. 
With one more repetition we observe that $x^*Rx_n$, so $x_n\sees\a_0$ by \eqref{Clast}, hence $x_nRx_{n+1}$ for some $x_{n+1}\in\a_0$, thus $x_{n+1}\models\ph_0$.
But now applying Lemma \ref{Rpath} to the $R$-path $x_0,\dots ,x_n,x_{n+1}$ we conclude that 
$x_0\models\P_n(\Vec{\ph}{n})$.
Since $x_0\models\neg\P_n(\Vec{\ph}{n})$ by  \eqref{con},  this is a contradiction, 
 forcing us to conclude that $C^*$ cannot have more that $n$  elements, and
 completing the proof of Lemma \ref{C*}.  
\end{proof}

Observe that in this lemma, if $R$ is reflexive then $x^*\sees\ab{x^*}\in C$ and so $\ab{x^*}\in C^*$. If also $n=1$ it follows that $C^*=\{\ab{x^*}\}$ and
\begin{equation} \label{vlast}
\text{if $x^*Ry$ and $\ab{y}\in C$, then $\ab{y}=\ab{x^*}$. }
\end{equation}
When \eqref{vlast} holds, $\ab{x^*}$ was called  \emph{virtually last in $C$} by \citet[II.3]{sege:essa71}, who showed that if $\M$ is a model of  S4Grz, then every cluster of $\M_\Phi$ contains a virtually last element. He then transformed $\M_\Phi$ into a semantically equivalent partially ordered model by replacing each cluster by an arbitrary linear ordering of its elements ending with a virtually last element. 

If $n=0$, then $C^*$ is empty, and hence by \eqref{Clast},  $x^*Ry$ implies $\ab{y}\notin C$. 
\citet[II.2]{sege:essa71} showed that if $\M$ is a model of  GL, then every cluster of $\M_\Phi$ contains an  element $\ab{x^*}$ with $x^*$ having this property . He then transformed $\M_\Phi$ into a semantically equivalent strictly partially ordered (i.e.\ transitive and irreflexive) model by replacing each cluster by an arbitrary strict linear ordering of its elements ending with $\ab{x^*}$. 

Lemma \ref{C*}  thus encompasses Segerberg's analysis for S4Grz and GL.
We now proceed to use the lemma to transform $\M_\Phi$ into an equivalent model of circumference at most $n$.
For each $R_\Phi$-cluster $C$, choose and fix a  point $x^*$ and associated set $C^*\sub C$ as given by the lemma.
We will call $x^*$ the \emph{critical point for $C$}.
 Then we define a subrelation $R'$ of $R_\Phi$ to refine the structure of each $R_\Phi$-cluster $C$ by decomposing it into the subset $C^*$ as an $R'$-cluster together with a degenerate $R'$-cluster $\{\a\}$ for each $\a\in C-C^*$.
These singleton clusters all have $C^*$ as an $R'$-successor but are $R'$-incomparable  with each other. So the structure replacing $C$ looks like
$$
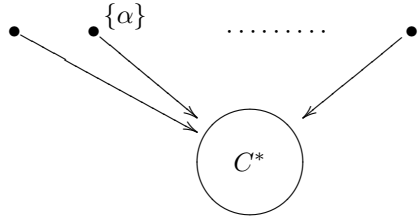
\xymatrix{
*{\bullet} \ar[drr]^<{}  &*{\bullet} \ar[dr]^<{\textstyle\{\a\}}  & {\qquad\cdots\cdots\cdots}   &*{\ \bullet^{}} \ar[dl]   \\
& &*{\xy ;<1pc,0pc>:\POS(0,0) +(0,-1.5)*+{C^*}*\cir<20pt>{} \endxy}  &{\hspace{-2.3cm}}
}
$$
with the  bullets being the degenerate $R'$-clusters determined by the  points of $C-C^*$, and the large circle representing $C^*$.  All elements of $W_\Phi$ that $R_\Phi$-precede $C$ continue to $R'$-precede all members of $C$, while elements of $W_\Phi$ that come $R_\Phi$-after $C$ continue to come $R'$-after all members of $C$. Doing this to each cluster of $(W_\Phi,R_\Phi)$ produces a new \emph{transitive} frame $(W_\Phi,R')$ with $R'\sub R_\Phi$.

$R'$ can be more formally defined on $W_\Phi$   by specifying, for all $\a,\b\in W_\Phi$, that $\a R'\b$ iff  either
\begin{itemize}
\item 
$\a$ and $\b$ belong to different $R_\Phi$-clusters and $\a R_\Phi \b$; \enspace or
\item
$\a$ and $\b$ belong to the same $R_\Phi$-cluster $C$ and  $\b\in C^*$.
\end{itemize}
 Thus every element of $C$ is $R'$-related to every element of $C^*$, and the restriction of $R'$ to $C$ is  equal to the relation
 $
 C\times C^*.
 $
So we could also define $R'$ as the union of  these relations $C\times C^*$ for all $R_\Phi$-clusters $C$, plus all inter-cluster instances of $R_\Phi$. If $C$ is $R_\Phi$-degenerate, then $C^*=\emptyset$ by  Lemma \ref{C*}, and so 
$C\times C^*=\emptyset$.  If $C$ is non-$R_\Phi$-degenerate, then the restriction of $R_\Phi$ to $C$ is $C\times C$, extending $C\times C^*$. This implies that $R'$ is a subrelation of $R_\Phi$ on $W_\Phi$.

Note that if $C^*$ is empty, then $C-C^*=C\ne\emptyset$, and all members of $C$ are $R'$-irreflexive. In that case $C$ is replaced in the new frame $(W_\Phi,R')$ by a non-empty set of degenerate $R'$-clusters. In the case $n=0$,  by Lemma \ref{C*} every $R_\Phi$-cluster $C$ has empty $C^*$, and so  $(W_\Phi,R')$ consists entirely of $R'$-irreflexive points and therefore has circumference 0.  In the alternative case  $n\ge 1$, any non-degenerate $R'$-cluster will have the form $C^*$ for some $R_\Phi$-cluster $C$, and so have at most $n$-elements. Since any $R'$-cycle is included in a non-degenerate $R'$-cluster, it follows that all $R'$-cycles have length at most $n$ and  $(W_\Phi,R')$ has circumference at most $n$. So in any case the finite transitive frame $(W_\Phi,R')$ validates $\C_n$ by Theorem \ref{sound}.

Now put $\M'=(W_\Phi,R',V_\Phi)$, a model differing from  $\M_\Phi$ only  in that $R'$ replaces $R_\Phi$. We show that replacing $\M_\Phi$ by $\M'$ leaves the truth relation unchanged for any formula $\ph\in\Phi$: for all $x\in W$,
\begin{equation} \label{truth}
\text{$\M',\ab{x}\models \ph$\enspace iff\enspace$\M_\Phi,\ab{x}\models \ph$. }
\end{equation}
The proof of this proceeds by induction on the formation of $\ph$. If $\ph$ is a variable, then \eqref{truth} holds because $\M'$ and $\M_\Phi$ have the same valuation $V_\Phi$. The induction cases of the Boolean connectives are standard. Now make the induction hypothesis on $\ph$  that  \eqref{truth} holds for all $x\in W$, and suppose $\di \ph\in\Phi$. If 
$\M',\ab{x}\models \di \ph$, then $\ab{x}R'\ab{y}$ and $\M',\ab{y}\models \ph$ for some $y$. Then $\ab{x}R_\Phi\ab{y}$ as $R'\sub R_\Phi$, and $\M_\Phi,\ab{y}\models \ph$ by induction hypothesis. Hence  $\M_\Phi,\ab{x}\models \di \ph$.

Conversely, assume $\M_\Phi,\ab{x}\models \di \ph$. Then $\M,x\models\di \ph$ by the Filtration Lemma \eqref{filtlem}.
Let $C$ be the $R_\Phi$-cluster of $\ab{x}$, and $x^*$ be the chosen critical point for $C$ fulfilling Lemma \ref{C*}. Then 
$\ab{x}$ and $\ab{x^*}$ both belong to $C$, so $\M,x^*\models\di \ph$ by \eqref{same}. Hence there is some $y\in W$ with $x^*Ry$ and $\M,y\models\ph$.
Then  $\M_\Phi,\ab{y}\models \ph$ by the Filtration Lemma \eqref{filtlem}, so $\M',\ab{y}\models \ph$ by induction hypothesis. If $\ab{y}\in C$, then $\ab{y}\in C^*$ by  \eqref{Clast}, so then $\ab{x}R'\ab{y}$ by definition of $R'$ since  $\ab{x}\in C$.
But if $\ab{y}\notin C$, then as $\ab{x^*}R_\Phi\ab{y}$ by \eqref{presR}, and so $\ab{x}R_\Phi\ab{y}$, the $R_\Phi$-cluster of $\ab{y}$ is strictly $R_\Phi$-later than $C$, and again $\ab{x}R'\ab{y}$ by definition of $R'$.
So in any case we have $\ab{x}R'\ab{y}$ and  $\M',\ab{y}\models \ph$, which gives  $\M',\ab{x}\models \di \ph$.
That completes the inductive case for $\di \ph$, and hence proves that \eqref{truth} holds for all $\ph\in \Phi$.

\begin{theorem}  \label{compl}
For  all $n\ge 0$ and any formula $\ph$ the following are equivalent.
\begin{enumerate}[\rm 1.]
\item 
$\ph$ is a theorem of \emph{K4}$\C_n$.
\item
$\ph$ is valid in all transitive frames that have circumference at most $n$ and no strictly ascending chains.
\item
$\ph$ is valid in all finite transitive frames that have circumference at most $n$.
\end{enumerate}
\end{theorem}
\begin{proof}
1 implies 2: Let $L$ be the set of formulas that are valid in all transitive frames that have circumference at most $n$ and no strictly ascending chains. Then $L$ is a transitive normal logic that contains $\C_n$ by  Theorem \ref{sound}. Hence $L$ includes K4$\C_n$. 

2 implies 3: This follows immediately from the fact that a finite frame has no strictly ascending chains.

3 implies 1: Put $L=\mathrm{K4}\C_n$. Suppose 1 fails for $\ph$, i.e.\ $\ph$ is not a theorem of $L$.  Then there exists an $x\in W_L$ with $\ph\notin x$, hence $\M_L,x\not\models\ph$ by the Truth Lemma \eqref{truthlemma}.
In the above construction of a finite model $\M'$, let  $\M$ be $\M_L$, and $\Phi$ be the set of all subformulas of $\ph$. Then by \eqref{filtlem} and \eqref{truth},
$\M',\ab{x}\not\models\ph$.
But the frame of $\M'$ is finite, transitive and has circumference at most $n$. This shows that $\ph$ fails to be valid on such a frame, so 3 does not hold for $\ph$.
\end{proof}

This theorem yields an alternative proof that
$\mathrm{K4}\C_{n+1}\sub\mathrm{K4}\C_n$, since any formula valid in all finite transitive frames that have circumference at most $n+1$ is valid in all finite transitive frames that have circumference at most $n$. A transitive frame consisting of a single cycle of length $n+1$ will validate $\mathrm{K4}\C_{n+1}$ but not $\C_n$, showing that that
the logics $\{\mathrm{K4}\C_n:n\ge 0\}$ form a \emph{strictly} decreasing sequence   of extensions of K4.

\begin{corollary}
\emph{K4}$=\bigcap_{n\ge 0}\emph{K4}\C_n$.
\end{corollary}
\begin{proof}
K4 is a sublogic of K4$\C_n$ for all $n$. For the converse inclusion,
if $\ph$ is not a K4-theorem, then it is invalid in some finite transitive frame $\F$. If $n$ is the size of the largest cycle in $\F$, or 0 if there are no cycles, then $\F$ has circumference at most $n$, so by Theorem \ref{compl}, $\ph$ is not a K4$\C_n$-theorem.
\end{proof}

The proof of Theorem \ref{compl}  establishes something more. It gives a computable upper bound on the size of the falsifying model $\M'$, showing that $\ph$ is a K4$\C_n$-theorem iff it is valid in all finite transitive frames that have circumference at most $n$ and have size at most $2^k$, where $k$ is the number of subformulas of $\ph$. But it is decidable whether a given finite frame is transitive and has circumference at most $n$, so by
well-known arguments \citep[Section 6.2]{blac:moda01}, it follows that it is \emph{decidable} whether or not a given formula is a K4$\C_n$-theorem. 

A potential strengthening of $\C_n$ is to replace $\bo^*$ in its antecedent by $\bo$. But the resulting formula is valid in all finite frames validating K4$\C_n$, so is a theorem of that logic.

\section{Extensions of K4$\C_n$}  \label{extns}

We will now show how to apply and adapt the construction of  $\M'$ to obtain  finite-frame characterisations of various logics that extend K4$\C_n$.
\subsection*{Seriality}

The D-axiom $\di\top$ is valid in a frame iff its relation is \emph{serial}, meaning that every point has an $R$-successor:
$\forall w \exists y(xRy)$. The inclusion of this axiom in a logic ensures that its canonical model is serial, as each point  satisfies  $\di\top$.
We assume now that the transitive model $\M$ having $\M\models\C_n$  also has serial $R$.
We  use this to show that the subrelation $R'$ of $\M'$ is also serial.

Suppose that a point $\a\in W_\Phi$ has an  $R_\Phi$-cluster $C$ that is not $R_\Phi$-final, i.e.\ there is some cluster $C'$ with $CR_\Phi C'$ but not $C'R_\Phi C$. Then any $\b\in C'$ has $\a R'\b$ so is an $R'$-successor of $\a$. Alternatively, if $C$ is final, let $x^*$ be the critical point for $C$. There is a $y$ with $x^*Ry$,
as $R$ is serial. But then $\ab{x^*}R_\Phi\ab{y}$ by \eqref{presR}, and so $\ab{y}\in C$ as $\ab{x^*}\in C$ and $C$ is final. But then $\ab{y}\in C^*$ by \eqref{Clast}. Since every member of $C$ is $R'$-related to every member of $C^*$, we get that $\a R'\ab{y}$, completing the proof that $R'$ is serial and hence the frame of $\M'$ validates $\di\top$.

Since every cluster in a finite transitive frame has a successor cluster that is final, we see that such a frame is serial iff \emph{every final cluster is non-degenerate.}

From all this we can infer that  KD4$\C_n$, the smallest normal extension of K4$\C_n$ to contain $\di\top$, is sound and complete for validity in all finite  transitive frames that have circumference at most $n$ and every final cluster non-degenerate. 

\subsection*{S4$\C_n$}

S4 is the logic K4T, where T is the scheme $\bo\ph\to\ph$.
 A frame validates T iff its relation is reflexive, and the inclusion of T in a logic  ensures that the canonical frame  is reflexive. Assume that $n\geq 1$ and the model $\M\models\C_n$ as above has reflexive $R$, hence $R_\Phi$ is reflexive by \eqref{presR}. Thus no $R_\Phi$-cluster is degenerate.
We modify the definition of $R'$ to make it  reflexive as well, so that the frame of $\M'$ validates T. The change occurs in the case of an $R_\Phi$-cluster $C$ having $C\ne C^*$. Then instead of making the singletons $\{\a\}$ for $\a\in C-C^*$ be degenerate, we make them all into  \emph{simple} $R'$-clusters by requiring that $\a R'\a$. Formally this is done by adding to the definition of $\a R'\b$ the third possibility that
\begin{itemize}
\item 
$\a$ and $\b$ belong to the same $R_\Phi$-cluster $C$, and $\a=\b\in C-C^*$.
\end{itemize}
Equivalently, the restriction of $R'$ to $C$ is equal to $(C\times C^*) \cup\{(\a,\a):\a\in C-C^*\}$.
Since $R_\Phi$ is reflexive, this modified definition of $R'$ still has $R'\sub R_\Phi$, and that is enough to preserve the proof of the truth invariance result \eqref{truth}  for the modified model $\M'$.

From this it follows that S4$\C_n$, the smallest normal extension of K4$\C_n$ to contain the scheme T, is sound and complete for validity in all finite reflexive transitive frames that have circumference at most $n$.

We left out the case $n=0$ here because the addition of $\C_0$ to S4 results in  the inconsistent logic that has all formulas as theorems.
The logics $\{\mathrm{S4}\C_n:n\ge 1\}$ form a {strictly} decreasing sequence  whose intersection is S4.

\subsection*{Linearity}

K4.3 is the smallest normal extension of K4 that includes the scheme
\begin{equation}  \label{point3}
\bo(\ph\land\bo\ph\to\psi)\lor\bo(\psi\land\bo\psi\to\ph).
\end{equation}
A frame validates this scheme iff it is \emph{weakly connected}, i.e. satisfies
$$
\forall x\forall y\forall z(xRy\land xRz\to yRz\lor y=z\lor zRy).
$$
The canonical frame of any normal extension of K4.3 is weakly connected.

If a transitive weakly connected frame is \emph{point-generated}, i.e.\  $W=\{x\}\cup\{y\in W:xRy\}$ for some point $x\in W$, then the frame is  \emph{connected}: it satisfies
$$
\forall y\forall z( yRz\lor y=z\lor zRy).
$$
Such a connected frame can be viewed as a linearly ordered set of clusters.

Now let $L$ be a normal extension of K4.3$\C_n$ and $\ph$ a formula that is not a theorem of $L$. Then there is some $x\in W_L$ with $\ph\notin x$. Put $W=\{x\}\cup\{y\in W_L:xR_Ly\}$  and let  $\M=(W,R,V)$ be the  submodel of $\M_L$ based on $W$. Then   $R$ is transitive, and is connected since $R_L$ is weakly connected and $(W,R)$ is point-generated. Also the fact that $W$ is $R_L$-closed and $\M_L\models\C_n$ ensures that $\M\models\C_n$.
Take $\Phi$ to be the set of subformulas of $\ph$,  and $\M_\Phi$ to be the standard transitive filtration of $\M$ through $\Phi$. Then $\M_\Phi,\ab{x}\not\models\ph$.
Moreover, since $R$ is connected, it follows from  \eqref{presR}  that $R_\Phi$ is connected.

We now modify each $R_\Phi$-cluster $C$ to obtain a suitable model $\M'$ with circumference at most $n$. The way we did this for 
 K4$\C_n$ allows  some leeway in how we define the relation $R'$ on $C-C^*$. The minimal requirement to make the  construction work is that each member of $C-C^*$ forms a singleton cluster  that $R'$-precedes $C^*$. Instead of making these members of $C-C^*$ incomparable with each other, we could form them arbitrarily into a linear sequence under $R'$ that precedes $C^*$, getting a structure that looks like
$$
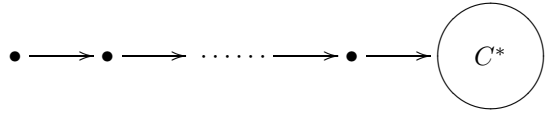
\xymatrix{
{\bullet}\ar[r]   &{\bullet}\ar[r] & {\ \cdots\cdots }\ar[r]^<{}  &{\bullet}\ar[r]
&{\xy ;<1pc,0pc>:\POS(0,0) +(0,0)*+{C^*}*\cir<20pt>{}\endxy} }
$$
The model $\M'$ will still satisfy \eqref{truth}, so will falsify $\ph$ at $\ab{x}$.
Since the original $R_\Phi$-clusters are linearly ordered by $R_\Phi$, this construction will make the  $R'$-clusters be linearly ordered by $R'$, with the non-degenerate ones being of size at most $n$.
Hence the frame underlying $\M'$ will be connected and validate  K4.3$\C_n$. This leads to the conclusion that
\begin{quote}
 the logic  K4.3$\C_n$ is sound and complete for validity in all finite transitive and connected frames that have circumference  at most $n$.
\end{quote}
The analysis of the scheme T can also be applied here to show further that 
\begin{quote}
for $n\ge 1$, the logic  S4.3$\C_n$ is sound and complete for validity  in all finite reflexive transitive and connected frames that have circumference at most $n$.
\end{quote}

\subsection*{Simple final clusters}

The McKinsey axiom is the scheme
\begin{equation} \label{M}
\text{M}: \quad   \bo\di\ph\to\di\bo\ph.
\end{equation}
The logic S4M is sometimes called S4.1.  Since S4M $\sub$ S4Grz \citep{sobo:fami64}, 
S4M$\C_1$ is just  S4$\C_1$. However   K4M$\C_1$  is stronger than  K4$\C_1$ (see below).

M is equivalent over K to $\di(\bo\ph\lor\bo\neg\ph)$, which we will make use of, and from which the seriality axiom $\di\top$ is derivable. A necessary and sufficient condition for a transitive frame to validate M is
$$
\forall x\exists y(xRy\ \&\ \forall z(yRz \text{ implies }y=z)).
$$
For a finite frame, this is equivalent to requiring that \emph{every final cluster is simple}.
\cite{sege:deci68} gave the first proof of the finite model property for S4M over frames satisfying this  requirement.  Here we use an adaptation of a proof method of \citet[Theorem 5.34]{chag:moda97} that works also for K4M.

Let $\M$ be a transitive model with $\M\models\text{M}$ and $\M\models\C_n$ for some $n\geq 1$.
Let $\Sigma$ be a finite set of formulas that is closed under subformulas (e.g.\ the set of subformulas of some non-theorem of K4M$\C_n$). Let $\Phi$ be the closure under subformulas of
$\{\bo\ph,\bo\neg\ph:\ph\in\Sigma\}$. Then $\Phi$ is still finite and closed under subformulas, so the finite filtration $\M_\Phi$  can be constructed as previously. We show that any $R_\Phi$-final cluster  $C$ is simple. To show that $C$ has one element, it  is enough to show that all members of $C$ satisfy the same formulas from $\Phi$ in $\M_\Phi$. Since all members of the same cluster satisfy the same formulas of the form $\bo\psi$ in any model, the problem reduces to the case of a formula $\ph\in\Sigma$. Take any $x$ with $\ab{x}\in C$. Then $\M,x\models\di(\bo\ph\lor\bo\neg\ph)$ as $\M\models\text{M}$. Hence there exists $y$ with $xRy$ and $\M,y\models\bo\ph\lor\bo\neg\ph$. Then $\ab{x}R_\Phi\ab{y}$, so $\ab{y}\in C$ as $C$ is final. As 
$\bo\ph,\bo\neg\ph\in\Phi$, and one of them is true at $y$  in $\M$, then one of them is true at $\ab{y}$  in $\M_\Phi$. Hence either $\ph$ is true at every member of $C$, or is false at every member of $C$.

That proves that $C$ is a singleton, so $\ab{x}=\ab{y}$ where $y$ is as above, hence $\ab{x}R_\Phi\ab{x}$ and $C=\{\ab{x}\}$is a simple $R_\Phi$-cluster as claimed. Moreover,  $x$ fulfils the requirements to be the critical point $x^*$ of Lemma \ref{C*}, so $C^*=C$, and $C$ becomes a  simple $R'$-cluster in the model $\M'$. We see that the final clusters of $\M'$ are just the final clusters of $\M_\Phi$, which are all simple in $\M'$. 
This leads to the conclusion that for $n\geq 1$,
\begin{quote}
K4M$\C_n$ is sound and complete for validity in all finite transitive  frames that have circumference at most $n$ and all final clusters simple, while  S4M$\C_n$ has the finite model property for the subclass of these frames that are reflexive.
\end{quote}
We also see that K4M$\C_1=$K4D$\C_1$, since the final clusters in a finite K4D$\C_1$-frame are singletons by $\C_1$-validity, and  are non-degenerate by seriality, so are simple. 

\subsection*{Degenerate final clusters}

Substituting $\bot$ for the variable in the L\"ob axiom produces a formula that is equivalent over K to
\begin{equation*} 
\text{\ws}: \quad   \bo\bot\lor\di\bo\bot.
\end{equation*}
A \emph{dead end} in a frame is an element $x$ that has no successor, i.e.\ $\{x\}$ is a degenerate final cluster. The condition for a model to satisfy \ws, or equivalently for its underlying frame to validate \ws, is that every point is either a dead end or has a successor that is a dead end. In a finite transitive frame, this condition is equivalent to requiring that \emph{every final cluster is degenerate.}

Assume now that a transitive model $\M$ having $\M\models\C_n$  also has $\M\models\ws$. The simplest way to ensure that our modified filtration $\M'$ also satisfies \ws\ is to include $\di\bo\bot$ and its subformulas in the finite set $\Phi$. Then using \eqref{filtlem} and \eqref{truth} we can infer that $\M'\models\ws$.
This leads to the conclusion that for $n\geq 1$,
\begin{quote}
K4\ws$\C_n$ is sound and complete for validity in all finite transitive  frames that have circumference at most $n$ and all final clusters degenerate.
\end{quote}

\section{Models on Irresolvable Spaces}  \label{sectop}

There is a topological semantics that interprets modal formulas as subsets of a topological space $X$, interpreting $\bo$ by the interior operator $\int_X$ of $X$ and $\di$ by its  closure operator $\cl_X$. This approach is known as \emph{C-semantics} \citep{bezh:some05} after the interpretation of $\di$ as \emph{C}losure. 
It contrasts with \emph{$d$-semantics}, which we will describe later.

A \emph{topological model} $\M=(X,V)$ on $X$ is given by a valuation $V$ assigning a subset $V(p)$ of $X$ to each variable $p$. A truth set $\M_C(\ph)$ is then defined by induction on the formation of an arbitrary formula $\ph$ by letting $\M_C (p)=V(p)$,  interpreting the Boolean connectives by the corresponding Boolean set operations, and putting $\M_C(\bo\ph)=\int_X(\M_C(\ph))$. Then $\M_C(\di\ph)=\cl_X(\M_C(\ph))$. 
A truth relation $\M,x\models_C\ph$ is defined to mean that $x\in \M_C(\ph)$. Defining an \emph{open neighbourhood} of $x$ to be any open subset of $X$ that contains $x$, we have

\begin{itemize}
\item $\M,x\models_C\bo\varphi$ iff there is an open neighbourhood  of $x$ included in $ \M_C(\ph)$,

\item $\M,x\models_C\di\varphi$ iff every open neighbourhood of $x$ intersects $ \M_C(\ph)$.
\end{itemize}
$\ph$ is \emph{$C$-true in $\M$}, written $\M\models_C\ph$, if $\M,x\models_C\varphi$ for all $x\in X$; and is \emph{C-valid in space $X$}, written $X\models_C\ph$, if it is $C$-true  in all models on $X$. 
For any space $X$, the set $\{\ph:X\models_C \ph\}$ of all formulas $C$-valid in $X$ is a normal logic including S4, called the \emph{C-logic of X}.

A  frame $(W,R)$  has the \emph{Alexandroff topology }on $W$ in which the open sets $O\sub W$ are those that are  \emph{up-sets} under $R$, i.e.\ if $x\in O$ and $xRy$ then $y\in O$. Call the resulting topological space $W_R$. If $R$ is a quasi-order, the interior operator $\int_R$ and closure operator $\cl_R$ of $W_R$ turn out to be given by
\begin{align}
\int\nolimits_R Y&=\{x\in W:\forall y(xRy \text{ implies } y\in Y)\}, \nonumber
\\
\cl\nolimits_R Y&=R^{-1}Y=\{x\in W:\exists y(xRy\in Y)\},  \label{CR}
\end{align}
which are the same operations that interpret $\bo$ and $\di$ in a model $\M=(W,R,V)$ in the Kripkean sense of Section \ref{secmodels}.
Such a  quasi-ordered model   gives rise to the topological model $\M_R=(W_R,V)$.  These two models are semantically equivalent in the sense that
\begin{equation*}   \label{semequiv}
\text{$\M,x\models\ph$\quad iff\quad$\M_R,x\models_C\ph$},
\end{equation*}
for all $x\in W$ and all formulas $\ph$. It follows, in terms of validity, that
\begin{equation}   \label{valequiv}
\text{$(W,R)\models\ph$\quad iff\quad$ W_R\models_C\ph$}.
\end{equation}
This can be used to show that the set of formulas that are $C$-valid in all topological spaces is exactly S4, a result due to \cite{mcki:theo48}, the originators of this kind of topological semantics.

For $n\geq 2$, a  space is called \emph{$n$-resolvable} if it has $n$ pairwise disjoint dense subsets \citep{ecke:reso97}. Here $Y$ is dense in $X$ when $\cl_X Y=X$. Since any superset of a dense set is dense, $n$-resolvability is equivalent to $X$ having a partition into $n$ dense subsets. $X$ is \emph{$n$-irresolvable} if it is not $n$-resolvable, and is  \emph{hereditarily $n$-irresolvable} if every non-empty subspace of $X$ is $n$-irresolvable. 

It is known that S4Grz is the $C$-logic determined by the class of hereditarily $2$-irresolvable spaces (the prefix $n$- is usually omitted when $n=2$). This follows from results of \citealt{esak:diag81} and \citealt{bezh:scat03}, as explained in \cite[pp.~253-4]{bent:moda07}. We will show that  in general S4$\C_n$ is characterised by $C$-validity in all hereditarily $n+1$-irresolvable spaces.

Note that by general topology, the closure operator $\cl_S$ of a subspace $S$ of $X$ satisfies $\cl_S Y= S\cap\cl_X Y$.
Hence a set $Y\sub S$ is dense in $S$ iff $S\sub\cl_X Y$.

\begin{theorem} \label{soundresolv}
For $n\geq 1$,  a topological space $X$ is hereditarily $n+1$-irresolvable iff\/ $X\models_C \C_n$.
\end{theorem}

\begin{proof}
Write $\int$ and $\cl$ for the  interior and closure operators of $X$.
Suppose first that $X$ is not hereditarily $n+1$-irresolvable. Then there is a non-empty subspace $S$ of $X$ that has $n+1$ pairwise disjoint dense subsets, say $S_0,\dots, S_n$. Take variables $p_0,\dots, p_n$ and let $\M$ be any model on $X$ for which $\M_C(p_i)=S_i$ for $i\leq n$. The disjointness of the $S_i$'s ensures that the formula  $\D_n(\Vec{p}{n})$ is $C$-true in $\M$ at every point of $X$, hence so is $\bo^*\D_n(\Vec{p}{n})$.

The truth set $\M_C(P_n(\Vec{p}{n}))$  is
$$
\cl(S_1\cap\cl(S_2\cap\cdots \cap\cl(S_n\cap\cl S_0))\cdots).
$$
But by the given density we have that $S_i\sub S\sub \cl S_j$ for all $i,j\leq n$. Thus 
$\cl(S_n\cap\cl S_0)=\cl S_n$, hence
$$
\cl(S_{n-1}\cap \cl(S_n\cap\cl S_0))=\cl(S_{n-1}\cap \cl S_n)=\cl S_{n-1},
$$
etc. Iterating this calculation we conclude that  $\M_C(\P_n(\Vec{p}{n}))=\cl S_1$.
Since $S_0\sub \cl S_1$, it follows that the formula 
$p_0\land\neg\P_n(\Vec{p}{n})$ is $C$-false at every point of $X$, hence so is
$\di(p_0\land\neg\P_n(\Vec{p}{n}) )$.
But  $\M_C(\di p_0)=\cl S_0$, so
$$
\bo\nolimits^*\D_n(\Vec{p}{n}) \to(\di p_0 \to
\di(p_0\land\neg\P_n(\Vec{p}{n}) )
$$
is $C$-false at every point of $\cl S_0$. Since $\emptyset\ne S\sub \cl S_0$, we conclude  that $\C_n$ is not $C$-valid in $X$.

That proves one direction of the theorem. For the other direction,
suppose that some instance
\begin{equation}  \label{Cninst}
\bo\nolimits^*\D_n(\Vec{\ph}{n}) \to(\di \ph_0 \to
\di(\ph_0\land\neg\P_n(\Vec{\ph}{n}) )
\end{equation}
of $\C_n$ is not $C$-valid in $X$, so is false at some point $x$ in some model $\M$ on $X$.
Put $A_i=\M_C(\ph_i)$ for all $i\leq n$ and $D_n=\bigcap_{ i<j\le n}-(A_i\cap A_j)$. Define
$$
P_n=  \cl(A_1\cap\cl(A_2\cap\cdots \cap\cl(A_n\cap\cl A_0))\cdots).
$$
Then $P_n=\M_C(\P_n)$, and the $C$-falsity of  \eqref{Cninst} at $x$ implies that $x\notin \cl(A_0-P_n)$ while   $x\in\cl A_0$ and $x\in\int D_n$. The latter implies that  there is an open neighbourhood $B$ of $x$ with $B\sub D_n$. 
Hence the sets $\{B\cap A_j:j\leq n\}$ are pairwise disjoint.

Since $x\notin \cl(A_0-P_n)$
there is an open neighbourhood $B'$ of $x$ with
\begin{equation}  \label{AminusP}
B'\cap(A_0-P_n)=\emptyset.
\end{equation}
Let $O=B\cap B'$, another  open set containing $x$. Put $S_0=O\cap A_0$ and for $1\leq i < n$, 
$$
S_i=  O\cap A_i\cap\cl(A_{i+1}\cap\cdots \cap\cl(A_n\cap\cl A_0))\cdots),
$$
while $S_n=O\cap A_n\cap\cl A_0$. 
Define $S=S_0\cup\cdots \cup S_n$. We will show that $S_i\sub\cl S_j$ for all $i,j\leq n$. This implies that for each $j\leq n$ we have $S\sub\cl S_j$, so $S_j$ is dense in the subspace $S$. But $S_j\sub O\cap A_j\sub B\cap A_j$, so the $S_j$'s are also pairwise disjoint.
Since $x\in\cl A_0$ we get $O\cap A_0\ne\emptyset$, i.e.\ $S_0\ne\emptyset$. Hence as $S_0\sub \cl S_j$, we get $S_j\ne\emptyset$ for all $j\leq n$. So the $S_j$'s form $n+1$ distinct pairwise disjoint dense subsets of $S$, showing that $S$ is $n+1$-resolvable and therefore $X$ is not hereditarily $n+1$-irresolvable.

It remains to prove that $S_i\sub\cl S_j$ in general. Define a binary relation $\rho$ on subsets of $X$ by putting $Y\rho Z$ iff $Y\sub\cl Z$. By closure algebra $\rho$ is transitive, since $Y\rho Z\rho Z'$ implies
$Y\sub\cl Z\sub\cl\cl Z'=\cl Z'$.
It is enough to prove that the $S_j$'s form a $\rho$-cycle, i.e.\ $S_0\rho S_1\rho\cdots\rho S_n\rho S_0$, for then the transitivity of $\rho$ ensures that $S_i\rho S_j$ for all $i,j\leq n$, as required.

We use the general fact that, since $O$ is open, $O\cap\cl Z\sub\cl(O\cap Z)$ for any $Z$. 
For $1\leq i< n$, let
$$
Z_i=A_i\cap\cl(A_{i+1}\cap\cdots \cap\cl(A_n\cap\cl A_0))\cdots),
$$
and put $Z_n=  A_n\cap\cl A_0$.
Then  if $i\leq n$ we have $S_i=O\cap Z_i$, and if $i<n$ then   $Z_i= A_i\cap\cl Z_{i+1}$. Also $P_n=\cl Z_1$.

To show that $S_0\rho S_1$, note that  $S_0\sub B'\cap A_0$, so from \eqref{AminusP}, $S_0\sub P_n$.
Since $S_0\sub O$, so then 
$$
S_0\sub O\cap P_n=O\cap\cl Z_1\sub\cl(O\cap Z_1)=\cl S_1.
$$ 
 For $S_i\rho S_{i+1}$ when $1\leq i< n$,
 $$
 S_i=O\cap A_i\cap\cl Z_{i+1}\sub O\cap\cl Z_{i+1}\sub\cl(O\cap Z_{i+1})=\cl S_{i+1}.
 $$
 Finally, for $S_n\rho S_0$ we have 
 $S_n=O\cap A_n\cap\cl A_0\sub O\cap\cl A_0\sub\cl(O\cap A_0)=\cl S_0$.
 \end{proof}
 
We can now  clarify the relationship between circumference and irresolvability:
 
\begin{theorem}  \label{lemnresolv}
Let $(W,R)$ be a quasi-ordered frame, and  $n\geq 1$.    

\begin{enumerate} [\rm 1.]
\item
The  topological space $W_R$ is hereditarily $n+1$-irresolvable iff $(W,R)$  has circumference at most $n$ and has  no strictly ascending chains.
\item
If  $W$ is finite, then $W_R$ is hereditarily $n+1$-irresolvable iff $(W,R)$  has circumference at most $n$.
\end{enumerate}
\end{theorem}

\begin{proof}
This is a restatement of Theorem \ref{sound} with `$W_R$ is hereditarily $n+1$-irresolvable' in place of 
`$(W,R)\models\C_n$'.
But by Theorem \ref{soundresolv}, $W_R$ is hereditarily $n+1$-irresolvable iff  $W_R\models_C\C_n$, which by \eqref{valequiv} holds iff $(W,R)\models\C_n$.  
\end{proof}

Given our results on relational semantics for S4$\C_n$, and the equivalence with  $C$-semantics given by \eqref{valequiv}, it follows from the second part of Theorem \ref{lemnresolv} that any non-theorem of S4$\C_n$ is falsifiable in a $C$-model on a finite hereditarily $n+1$-irresolvable space. Hence S4$\C_n$ is complete for $C$-validity in (finite) hereditarily $n+1$-irresolvable spaces. Soundness follows from Theorem \ref{soundresolv}, which ensures that the $C$-logic of any  hereditarily $n+1$-irresolvable space includes S4$\C_n$.
 
 \begin{theorem}   \label{Cntopfmp}
For  all $n\ge 1$ and any formula $\ph$ the following are equivalent.
\begin{enumerate}[\rm 1.]
\item 
$\ph$ is a theorem of \emph{S4}$\C_n$.
\item
$\ph$ is $C$ valid in all hereditarily $n+1$-irresolvable topological spaces.
\item
$\ph$ is $C$-valid in all finite hereditarily $n+1$-irresolvable topological spaces of size at most $2^k$, where $k$ is the number of subformulas of $\ph$.
\end{enumerate}
\end{theorem}

Next we discuss the topological role of the McKinsey axiom M of \eqref{M}. It is $C$-valid in precisely those spaces for which every non-empty \emph{open} subspace is irresolvable \cite[Prop.~2.1]{bezh:scat03}. A point $x$ is \emph{isolated} in a space $X$  if $\{x\}$ is open in $X$. Such a point will belong to one cell of any partition and prevent any other cell from being dense, hence preventing resolvability. So a sufficient condition for $C$-validity of M is that the space is \emph{weakly scattered}, meaning that the set of isolated points is dense, since that ensures that any non-empty open subspace has an isolated point and therefore is irresolvable.
Now in a quasi-ordered frame $(W,R)$, if a final cluster is simple, then since it is open in the Alexandroff space $W_R$, its one element is isolated in $W_R$. If the frame is finite and validates M,  then every point is $R$-succeeded by a point whose cluster is final and therefore simple, implying that $W_R$ is weakly scattered. Combining this with the analysis behind Theorem \ref{Cntopfmp} and the relational characterisation of S4M$\C_n$ at the end of Section \ref{extns}, we conclude that

\begin{quote}
for $n\geq 2$, S4M$\C_n$ is sound and complete for $C$-validity in all (finite) weakly scattered spaces
that are hereditarily $n+1$-irresolvable.
\end{quote}

We turn now to \emph{$d$-semantics}, which interprets $\di$ by the \emph{derived set} operator $\de_X$ of a space $X$.
For a topological model $\M=(X,V)$ it generates truth sets $\M_d(\ph)$ that have
 $\M_d(\di\ph)=\de_X(\M_d(\ph))$,  the set of limit points of $\M_d(\ph)$. The truth relation $\M,x\models_d\ph$ now  means that $x\in \M_d(\ph)$. Defining a \emph{punctured neighbourhood of $x$} to be any set of the form $O-\{x\}$ where $O$ is an open neighbourhood of $x$, we have

\begin{itemize}
\item $\M,x\models_d\di\varphi$ iff every punctured neighbourhood of $x$ intersects $ \M_d(\ph)$,

\item $\M,x\models_d \bo\varphi$ iff there is a punctured neighbourhood  of $x$ included in $ \M_d(\ph)$.
\end{itemize}

A formula $\ph$ is \emph{$d$-valid in $X$}, written $X\models_d\ph$, if  it is $d$-true in every model on $X$.
The \emph{$d$-logic} $\{\ph:X\models_d\ph\}$ of $X$  is a normal logic which always includes the weak transitivity scheme w4, because the operator $\de_X$  always has $\de_X\de_X Y\sub Y\cup\de_X Y$ for all $Y\sub X$.
 The condition for $X$ to $d$-validate the $\di$-version $\di\di\ph\to\di\ph$ of scheme 4 is that in general $\de_X\de_X Y\sub\de_X Y$. This is equivalent to requiring that the derived set $\de_X\{x\}$ of any singleton is closed. A space with this property is called \emph{T$_D$}. In terms of strength, the T$_D$-property lies strictly between the  separation  properties T$_0$ and T$_1$ \citep{aull:sepa62}. 
 A T$_1$ space is one in which the derived set $\de_X\{x\}$ of any singleton is empty. 
 A T$_0$ space is one in which for any two distinct points $x$ and $y$ there is an open set containing one of them and not the other, i.e.\ either $x\notin\cl\{y\}$ or $y\notin\cl\{x\}$. 
 
 Any hereditarily irresolvable space is T$_D$, as
in general  $\cl_X\{x\}=\{x\}\cup\de_X\{x\}$ with $x\notin \de_X\{x\}$ and $\{x\}$  dense in $\cl_X\{x\}$. So if $\cl_X\{x\}$ is irresolvable, then $\de_X\{x\}$ cannot be dense in $\cl_X\{x\}$, hence  $\cl_X\de_X\{x\}$ can only be  $\de_X\{x\}$, i.e.\  $\de_X\{x\}$ is closed.

In $C$-semantics there is no distinction in interpretation between $\bo$ and $\bo^*$, or between $\di$ and $\di^*$. Since $\int_X Y\sub Y\sub\cl_X Y$ we have $\M_C(\bo^*\ph)=\M_C(\bo\ph)$ and $\M_C(\di^*\ph)=\M_C(\di\ph)$.
On the other hand in $d$-semantics, $\di^*$ and $\bo^*$ serve to define
the closure and interior of  $\M_d(\ph)$.
Since in general $\cl_X Y=Y\cup\de_X Y$, we get $\M_d(\di^*\ph)=\cl_X(\M_d(\ph))$, and thus
  $\M_d(\bo^*\ph)=\int_X(\M_d(\ph))$. 
  
Since $\C_n$ defines the class of hereditarily $n+1$-irresolvable spaces under the $C$-semantics (Theorem \ref{soundresolv}),  by replacing every occurrence of $\bo$ and $\di$ in $\C_n$ by $\bo^*$ and $\di^*$ we get a formula that defines the class of hereditarily $n+1$-irresolvable spaces under the $d$-semantics. However, we can more simply use 
$\C_n$ itself:

\begin{theorem} \label{Cn*def}
For $n\geq 1$, a topological space $X$ has $X\models_d\C_n$ iff $X\models_d\C_n^*$ iff $X\models_C\C_n$
iff $X$ is hereditarily $n+1$-irresolvable.
\end{theorem}

\begin{proof}
$X\models_d\C_n$ implies $X\models_d\C_n^*$ because $\C_n^*$ is a tautological consequence of $\C_n$.

Conversely $X\models_d\C_n^*$ implies $X\models_d\C_n$, because if $X\models_d\C_n^*$ then the $d$-logic of $X$ includes $\C_n^*$ and w4 (since w4 is $d$-valid in all spaces), so by Theorem \ref{w44}(2) this $d$-logic includes $\C_n$, hence $X\models_d\C_n$.

We already saw in Theorem \ref{soundresolv} that 
$X$ is hereditarily $n+1$-irresolvable iff $X\models_C \C_n$.
Next we show that $X\models_C\C_n$ implies $X\models_d \C_n^*$.
For, if  $X\not\models_d \C_n^*$, then
some instance of $\C_n^*$ of the form
\begin{equation}  \label{Cn*inst}
\bo\nolimits^*\D_n(\Vec{p}{n}) \to(\di p_0 \to
\di^*(p_0\land\neg\P_n(\Vec{p}{n}) )
\end{equation}
 is not $d$-valid in $X$, so is $d$-false at some point $x$ in some model $\M=(X,V)$.
Put $A_i=V(p_i)$ for all $i\leq n$ and $D_n=\bigcap_{ i<j\le n}-(A_i\cap A_j)$. Then
\begin{align*}
\M_d(\P_n)  &=  \de(A_1\cap\de(A_2\cap\cdots \cap\de(A_n\cap\de A_0))\cdots), \enspace \text{and}
\\
\M_C(\P_n) &=  \cl(A_1\cap\cl(A_2\cap\cdots \cap\cl(A_n\cap\cl A_0))\cdots).
\end{align*}
The $d$-falsity of  \eqref{Cn*inst} at $x$ implies that $x\in\int D_n$ and  $x\in\de A_0$ but
$
x\notin \cl(A_0-\M_d(\P_n) ).
$
As $\de Y\sub\cl Y$ in general,   $\M_d(\P_n) \sub \M_C(\P_n) $.   Hence as $\cl$ preserves set inclusion, 
$$
\cl(A_0-\M_C(\P_n) )\sub \cl(A_0-\M_d(\P_n) ).
$$
Therefore  $x\notin \cl(A_0-\M_C(\P_n) )$. But $x\in\int D_n$ and  $x\in\de A_0\sub\cl A_0$, and these facts together ensure that 
$$
\M,x\not\models_C
\bo\nolimits^*\D_n(\Vec{p}{n}) \to(\di p_0 \to
\di(p_0\land\neg\P_n(\Vec{p}{n}) ).
$$
Therefore $X\not\models_C \C_n$ as required.

Finally, to complete the cycle of implications we show that $X\models_d \C_n^*$ implies that  $X$ is  hereditarily $n+1$-irresolvable.
For if  $X$ is not hereditarily $n+1$-irresolvable, then there is a non-empty subspace $S$ of $X$ that has $n+1$ pairwise disjoint dense subsets, say $S_0,\dots, S_n$. Take variables $p_0,\dots, p_n$ and let $\M$ be any model on $X$ for which $\M_d(p_i)=S_i$ for $i\leq n$. The disjointness of the $S_i$'s ensures that the formula  $\D_n(\Vec{p}{n})$ is $d$-true in $\M$ at every point of $X$, hence so is $\bo^*\D_n(\Vec{p}{n})$.

The truth set  $\M_d(\P_n(\Vec{p}{n}))$ is
$$
\de(S_1\cap\de(S_2\cap\cdots \cap\de(S_n\cap\de S_0))\cdots).
$$
By the given density we have that $S_i\sub S\sub \cl S_j$ for all $i,j\leq n$.  But now we observe that as $\cl S_j=S_j\cup\de S_j$, and the $S_i$'s are pairwise disjoint,  this implies that
\begin{equation*}
\text{$S_i\sub \de S_j$ for all $i\ne j\leq n$.}
\end{equation*}
Thus 
$\de(S_n\cap\de S_0)=\de S_n$, hence
$$
\de(S_{n-1}\cap \de(S_n\cap\de S_0))=\de(S_{n-1}\cap \de S_n)=\de S_{n-1},
$$
etc. Iterating this calculation we conclude that  $\M_d(\P_n)=\de S_1$.
Since $S_0\sub \de S_1$, it follows that the formula 
$p_0\land\neg\P_n(\Vec{p}{n})$ is $d$-false at every point of $X$, hence so is
$\di^*(p_0\land\neg\P_n(\Vec{p}{n}) )$.
But  $\M_d(\di p_0)=\de S_0$, so
$$
\bo\nolimits^*\D_n(\Vec{p}{n}) \to(\di p_0 \to
\di^*(p_0\land\neg\P_n(\Vec{p}{n}) )
$$
is $d$-false at every point of $\de S_0$.
Since $\de S_0\ne\emptyset$ (e.g.\ from $\emptyset\ne S\sub \cl S_1$ we get $S_1\ne\emptyset$, but $S_1\sub\de S_0$),
 we conclude  that $X\not\models_d\C_n^*$.
\end{proof}

We saw earlier that every  hereditarily irresolvable space is T$_D$ and hence T$_0$. The converse is true for finite spaces.
The essential reason is that a finite T$_0$ space is \emph{scattered}, meaning that any non-empty subspace has an isolated point, which ensures that it is irresolvable.   The properties of being  hereditarily irresolvable, scattered, or T$_0$ are equivalent for finite spaces   \citep[Corollary 4.8]{bezh:scat03}.
Also, a finite space is T$_0$ iff it  is T$_D$  \citep[Corollary 5.1]{aull:sepa62}. So if $X$ is finite, in the case $n=1$ we can add `$X$ is scattered', `$X$ is T$_D$' and  `$X$ is T$_0$' to the list of equivalent conditions in Theorem \ref{Cn*def}.

The logic K4$\C_1$, in the form K4Grz$_\Box$, was shown by  \cite{gabe:topo04} to be complete for $d$-validity in hereditarily irresolvable spaces (see also \cite[]{bezh:k4gr10}). This was done by  showing that for any finite  K4Grz$_\Box$-frame $\F$ there exists an hereditarily irresolvable space $X$ having a mapping from $X$ onto $\F$ that ensures that $X\models_d\ph$ implies $\F\models\ph$. The construction makes $X$ infinite, so does not provide the finite model property for K4$\C_1$ under $d$-semantics. In fact none of the logics  K4$\C_n$ with $n\geq 1$  have this finite model property.
More generally, if a normal extension of K4  does have the finite model property under $d$-semantics, then it must be an extension of the G\"odel-L\"ob logic GL (which  K4$\C_n$ is not when $n\geq 1$). This is because a finite space that $d$-validates K4 is a finite T$_D$ space, so  is scattered, as explained in the previous paragraph.
But \cite{esak:diag81} showed that the L\"ob axiom is $d$-valid in precisely the scattered spaces.  
Thus any finite space $d$-validating K4 must $d$-validate GL.

For $n\geq 2$, the logic K4$\C_n$  is characterised by $d$-validity in all hereditarily $n+1$-irreducible T$_D$ spaces. This will be shown in another article.

\section{Generating Varieties of Algebras}

Modal formulas have algebraic models, and the algebraic models of a logic form a \emph{variety}, i.e.\ an equationally definable class. Our Theorem \ref{compl} can be converted into a demonstration that the variety $\V_n$ of algebraic models of K4$\C_n$ is generated by its finite members, and indeed generated by certain finite algebras constructed out of finite transitive frames of circumference at most $n$. This implies that every member of $\V_n$ is a homomorphic image of a subalgebra of a direct product of such finite algebras. But something stronger can be shown: every member of $\V_n$ is isomorphic to a subalgebra of an  ultraproduct of such finite algebras. A proof of this algebraic fact will now be given that makes explicit use of the filtration construction underlying Theorem \ref{compl} along with further logical analysis involving the universal sentences that are satisfied by members of $\V_n$.

We briefly review the modal algebraic semantics (a convenient reference for more information is
 \citealt[Chapter 5]{blac:moda01}).
A \emph{normal modal algebra} has the form $\A=(\B,f^\A)$ with $\B$ a Boolean algebra and $f^\A$ a unary operation on $\B$ that preserves all finite meets (including the empty meet as the greatest element 1).
A modal formula can be viewed as a term of the language of $\A$, treating its variables as ranging over the individuals of $\B$, with  $\neg$ and $\land$  denoting the complement and meet operation of $B$,
 $\top$  denoting 1 and $\bo$  denoting $f^\A$. If $\ph$ has its variables among $\Vec{p}{k-1}$, then it induces a $k$-ary term function $\ph^\A$ on $\A$. We say that
  $\ph$ is \emph{valid in $\A$} when $\A\models\ph\approx\top$, meaning that  the equation $\ph\approx\top$ is satisfied in $\A$ in the usual sense from equational logic that $\A\models(\ph\approx\top)[\vv{a}]$, i.e.\ $\ph^\A(\vv{a})=1$, for all $k$-tuples $\vv{a}$ of elements of $
  \A$. The satisfaction of any equation $\ph\approx\psi$ in $\A$ is expressible as the validity of a modal formula, since $\A\models\ph\approx\psi$ iff the formula 
  $\ph\leftrightarrow\psi$ is valid in $\A$, i.e.\ iff $\A\models(\ph\leftrightarrow\psi)\approx\top$.
  
  Identifying $\ph$ with the equation $\ph\approx\top$, we can now  legitimately write $\A\models\ph[\vv{a}]$ to mean that
  $\ph^\A(\vv{a})=1$.
  
  An algebra $\A$ will be called \emph{transitive} if it validates the scheme 4, i.e.\ the formula $\bo\ph\to\bo\bo\ph$ is valid in $\A$ for all  $\ph$. Now the set $L_\A$ of all modal formulas that are valid in $\A$ is a normal logic that is closed under  uniform substitution of formulas for variables, so for $\A$ to validate scheme 4 it is enough that it validates
  $\bo p\to\bo\bo p$ for a variable $p$, which amounts to requiring that $f^\A a\le f^\A f^\A a$ for every element $a$ of $A$.

 Each frame $\F=(W,R)$ has an associated algebra $\F^+=(\mathcal{P} W,[R])$, where $\mathcal{P} W$ is the Boolean set algebra of all subsets of $W$ and the unary operation $[R]$ on $\mathcal{P} W$ is defined by 
 $[R]X=\{x:\forall y(xRy \text{ implies }y\in X)\}$. $\F^+$ is called the \emph{complex algebra of $\F$}.
 A \emph{complex algebra} more generally is defined as one that is a subalgebra of some algebra of the form $\F^+$. 

 Given such a subalgebra $\A$ of $\F^+$, consider a model $\M=(\F,V)$ on $\F$ and a formula $\ph(\Vec{p}{k-1})$ such that $V(p_i)\in\A$ for all $i<k$. Then it can be shown that
 $$
 \ph^\A(\Vec{p^\M}{k-1})=\ph^\M,
 $$
 where  $\psi^\M$ is the truth set $\{x:\M,x\models \psi\}$. From this it follows that
 \begin{equation} \label{truthsetA}
\A\models\ph[\Vec{p^\M}{k-1}] \quad\text{iff}\quad \M\models\ph.
\end{equation}
 Using this in the case that $\A=\F^+$ leads to a proof that a formula $\ph$ is valid in $\F$, i.e.\ $\M\models\ph$ for all models $\M$ on $\F$,
  iff it is valid in the normal modal algebra $\F^+$ in the sense defined here that $\F^+\models \ph\approx\top$
 (see \citealp[Prop.~5.24]{blac:moda01} for details of this analysis).
 
 The famous representation theorem of \citet{jons:bool51} showed that any normal modal algebra $\A$ is isomorphic to a complex algebra, i.e.\ there is a monomorphism $\A\rightarrowtail\F^+$ for some frame $\F$. Moreover they showed that certain equational properties are preserved in passing from $\A$ to $\F^+$. In particular they proved that if $\A$ is a transitive algebra then so is $\F^+$, and furthermore that this implies that the binary relation $R$ of $\F$ is transitive.

We use the standard symbols $\bH$, $\bS$, $\bP$, $\Pu$ for the class operations of closure under homomorphic images, isomorphic copies of subalgebras, direct products and ultraproducts respectively. A class of algebras is a variety iff it is closed under $\bH$, $\bS$ and $\bP$. The smallest variety containing a given class of algebras $\K$ is $\bH\bS\bP \K$, which is called the variety \emph{generated by} $\K$. It is the class of all models of the  \emph{equational theory of} $\K$, which is the set of all equations  satisfied by $\K$.

Let $\V_n$ be the variety of all algebras that validate all theorems of the logic $\mathrm{K4}\C_n$. A sufficient condition for membership of an algebra  $\A$ in $\V_n$ is that $\A$ is a normal modal algebra that validates the schemes 4 and $\C_n$.
This is because if the logic $L_\A$ comprising all modal formulas that are valid in $\A$  includes 4 and $\C_n$, then it includes $\mathrm{K4}\C_n$ since the latter is the smallest logic to include these schemes. As explained above, for $L_\A$ to include a scheme it suffices for it to contain a  variable instance of it. Thus $\V_n$ is defined by finitely many equations.

Let $\CC_n$ be the class of all finite transitive frames of circumference at most $n$, and
$\CC_n^+=\{\F^+:\F\in\CC_n\}$ the class of all complex algebras of members of $\CC_n$. Each $\F^+\in\CC_n^+$ validates 4 and $\C_n$, since $\F$ does, so $\CC_n^+\sub \V_n$. 
We then have
\begin{equation}\label{inclusions}
\bS\bP_U\CC_n^+\sub \bH \bS\bP\CC_n^+\sub\V_n.
\end{equation}
The first of these inclusions holds because the variety $ \bH \bS\bP\CC_n^+$ includes  $\CC_n^+$ and is closed under subalgebras and ultraproducts. The second holds because $\V_n$ is closed under $ \bH$, $ \bS$, and $\bP$. We will show  that both inclusions are equalities, so the three classes displayed in \eqref{inclusions} are identical. To prove this we need some background theory about the \emph{universal sentences} that are satisfied in $\V_n$.

A universal sentence in the language of modal algebras has the form $\forall \vv{p}\sig$, where  the formula $\sig$ is quantifier-free, so is a Boolean combination of equations, and $\forall \vv{p}$ is a sequence of universal quantifiers including  those for all the variables of $\sig$.  The following result is a standard fact in the model theory of universal sentences.

\begin{lemma}\label{univ}
If every universal sentence satisfied by a class $\K$ of algebras is satisfied by algebra $\A$, then $\A$ is embeddable into an ultraproduct of members of $\K$, i.e.\ $\A\in\bS\Pu \K$.
\end{lemma}
\begin{proof}
See \citealp[Section V.2]{burr:univ81}, especially the proof of Theorem 2.20.
\end{proof}
This result implies that $\bS\Pu \K$ is the class of all models of the \emph{universal theory of} $\K$, which is the set of all universal sentences satisfied by $\K$.

\begin{theorem}
For any $n\ge 0$,  $\V_n=\bH \bS\bP\CC_n^+= \bS\bP_U\CC_n^+$.
\end{theorem}
\begin{proof}
By \eqref{inclusions} it suffices to show that any member of $\V_n$ belongs to $\bS\bP_U\CC_n^+$.  So take any 
$\A\in\V_n$. To show that $\A\in \bS\bP_U\CC_n^+$, it is enough by  Lemma \ref{univ} to show that every universal sentence satisfied by  $\CC_n^+$ is satisfied by $\A$. We prove the contrapositive of this.

Let $\forall \vv{p}\sig$ be a universal sentence that is not satisfied by $\A$. We will show it is not satisfied by (some member of) $\CC_n^+$. We can suppose that $\sig$ is in conjunctive normal form, so that the sentence has the shape
 $\forall \vv{p}(\bigwedge_{i<m}\sig_i)$, where each conjunct $\sig_i$ is a disjunction of equations and negations of equations.
The sentence is then equivalent to  $\bigwedge_{i<m}\forall \vv{p}\sig_i$, so there must be some $i<m$ such that 
$\A\not\models\forall \vv{p}\sig_i$. This $\sig_i$ has the form
$$
\Big(\bigvee\nolimits_{h<k}\ph_h\Big) \lor \Big( \bigvee\nolimits_{j<l}\neg\psi_j \Big),
$$
where the $\ph_h$'s and $\psi_j$'s are equations. Then for some interpretation $\vv{a}$ in $\A$ of the list of variables $\vv{p}$, we have
\begin{equation}  \label{falsetrue}
\A\not\models\ph_h[\vv{a}] \quad\text{and}\quad \A\models\psi_j[\vv{a}] 
\end{equation}
for all $h<k$ and $j<l$. We  identify the equations $\ph_h,\psi_j$ with modal formulas and switch to dealing with modal models.

By the J\'onsson-Tarski representation theory we can take $\A$ to be a subalgebra of $\F^+$ for some  frame $\F=(W,R)$ with transitive $R$. So each member of $\A$ is a subset of $W$. Let $\M$ be any model on $\F$ such that for any variable 
$q$ the truth set $q^\M$ is in $\A$, and if
$q$ occurs in the list $\vv{p}$, then  $q^\M$ is the corresponding entry from $\vv{a}$. Then every truth set of $\M$ is in $\A$, and by \eqref{falsetrue} and \eqref{truthsetA}, for all  $h<k$ and $j<l$ we get
\begin{equation}  \label{ineqinM}
\M\not\models\ph_h \quad\text{and}\quad \M\models\psi_j.
\end{equation}
Now we apply the filtration construction of Section \ref{fmpn} to the transitive model $\M$, taking  $\Phi$ to be the closure under subformulas of the set $\{\ph_h,\psi_j:h<k\ \&\  j<l\}$, so  $\Phi$ is finite.
The algebra $\A$ belongs to $\V_n$ and so validates $\mathrm{K4}\C_n$. Hence by \eqref{truthsetA} we get
$\M\models\mathrm{K4}\C_n$ as required.

The filtration construction produces a finite transitive model $\M'=(W_\Phi,R',V_\Phi)$ of circumference at most $n$, such that for all $\ph\in\Phi$ and $x\in W$,
\begin{equation*} 
\text{$\M,x\models \ph$ \quad iff \quad $\M',\ab{x}\models \ph$. }
\end{equation*}
 by \eqref{filtlem} and \eqref{truth}. Since the map $x\mapsto\ab{x}$ is surjective from $W$ to $W'$, this implies that
 $$
\text{$\M\models \ph$ \quad iff \quad $\M'\models \ph$ }
$$
for all $\ph\in\Phi$. Applying this to \eqref{ineqinM} gives that
for all  $h<k$ and $j<l$ we have
\begin{equation}  \label{Mdashhj}
\M'\not\models\ph_h \quad\text{and}\quad \M'\models\psi_j.
\end{equation}
Now let $\A'$ be the complex algebra $(W_\Phi,R')^+$. Since the frame $(W_\Phi,R')$ belongs to $\CC_n$ we have $\A'\in\CC_n^+$.  Interpret each variable $q$ from $\vv{p}$ as the element $q^{\M'}$ of $\A'$ to get an interpretation $\vv{b}$ of $\vv{p}$ in $\A'$. Combining \eqref{truthsetA} for $\A'$ and $\M'$  with \eqref{Mdashhj} then gives 
$$
\A'\not\models\ph_h[\vv{b}] \quad\text{and}\quad \A'\models\psi_j[\vv{b}] 
$$
for all $h<k$ and $j<l$. Hence $\A'\not\models\forall \vv{p}\sig_i$, and so 
$\A'\not\models\forall \vv{p}(\bigwedge_{i<m}\sig_i)$.

Altogether we have shown that if any universal sentence is falsifiable in $\A$ then it is falsifiable in some member $\A'$ of 
$\CC_n^+$, hence if it is satisfied by  $\CC_n^+$ then it is satisfied by $\A$. As explained at the beginning, this is enough to prove the theorem.
\end{proof}
It follows from this theorem that any member of $\V_n$, i.e.\  any model of the equational theory of $\CC_n^+$, must be a model of the universal theory of $\CC_n^+$. This phenomenon is not special to the $\V_n$'s. It will occur for many varieties that can be shown to be generated by their finite members by the kind of methods we have used here.

{ 
\begin{subsubsection}*{\small Acknowledgement}
\small
I thank Ian Hodkinson and an anonymous referee for helpful comments and suggestions.
\end{subsubsection}
}


\end{document}